\pdfoutput=1
\documentclass[reqno, 11pt]{amsart}
\usepackage{geometry}
\makeatletter
\let\origsection=\section \def\section{\@ifstar{\origsection*}{\mysection}}
\def\mysection{\@startsection{section}{1}\z@{.7\linespacing\@plus\linespacing}{.5\linespacing}{\normalfont\scshape\centering\S}}
\makeatother

\usepackage{amsmath,amssymb,amsthm}
\usepackage{mathrsfs}
\usepackage{mathabx}\changenotsign
\usepackage{dsfont}
\usepackage{bbm}

\usepackage[open,openlevel=2,atend]{bookmark}

\usepackage[
	backend=biber,
	style=numeric,
	giveninits=true,
	url=true,
	isbn=false,
	maxbibnames=9,
	hyperref=auto,
	backref=true,
	]{biblatex}
\usepackage{csquotes}
\DeclareFieldFormat[misc]{date}{Preprint (#1)}
\renewcommand*{\bibfont}{\small}
\usepackage{xcolor}
\hypersetup{
    colorlinks,
    linkcolor={red!60!black},
    citecolor={green!60!black},
    urlcolor={blue!60!black}
}

\usepackage[T1]{fontenc}
\usepackage{lmodern}
\usepackage[babel]{microtype}
\usepackage[english]{babel}
\usepackage{verbatim}

\numberwithin{equation}{section}
\numberwithin{figure}{section}

\usepackage[shortlabels, inline]{enumitem}

\let\polishlcross=\l
\def\l{\ifmmode\ell\else\polishlcross\fi}

\let\emptyset=\varnothing
\let\setminus=\smallsetminus

\makeatletter
\def\moverlay{\mathpalette\mov@rlay}
\def\mov@rlay#1#2{\leavevmode\vtop{   \baselineskip\z@skip \lineskiplimit-\maxdimen
   \ialign{\hfil$\m@th#1##$\hfil\cr#2\crcr}}}
\newcommand{\charfusion}[3][\mathord]{
    #1{\ifx#1\mathop\vphantom{#2}\fi
        \mathpalette\mov@rlay{#2\cr#3}
      }
    \ifx#1\mathop\expandafter\displaylimits\fi}
\makeatother

\DeclareFontFamily{U}  {MnSymbolC}{}
\DeclareSymbolFont{MnSyC}         {U}  {MnSymbolC}{m}{n}
\DeclareFontShape{U}{MnSymbolC}{m}{n}{
    <-6>  MnSymbolC5
   <6-7>  MnSymbolC6
   <7-8>  MnSymbolC7
   <8-9>  MnSymbolC8
   <9-10> MnSymbolC9
  <10-12> MnSymbolC10
  <12->   MnSymbolC12}{}
\DeclareMathSymbol{\powerset}{\mathord}{MnSyC}{180}

\usepackage{tikz}
\usetikzlibrary{calc,decorations.pathmorphing}
\pgfdeclarelayer{background}
\pgfdeclarelayer{foreground}
\pgfdeclarelayer{front}
\pgfsetlayers{background,main,foreground,front}

\newcommand{\qedge}[7]{

	\ifx\relax#4\relax
		\def\qoffs{0pt}
	\else
		\def\qoffs{#4}
	\fi

	\def\qhedge{
		($#1+#3!\qoffs!-90:#2-#3$) --
		($#2+#1!\qoffs!-90:#3-#1$) --
		($#3+#2!\qoffs!-90:#1-#2$) -- cycle}

	\coordinate (12) at ($#1!\qoffs!90:#2$);
	\coordinate (13) at ($#1!\qoffs!-90:#3$);
	\coordinate (23) at ($#2!\qoffs!90:#3$);
	\coordinate (21) at ($#2!\qoffs!-90:#1$);
	\coordinate (31) at ($#3!\qoffs!90:#1$);
	\coordinate (32) at ($#3!\qoffs!-90:#2$);
	
	\def\nqhedge{
		(13) let \p1=($(13)-#1$), \p2=($(12)-#1$) in
			arc[start angle={atan2(\y1,\x1)}, delta angle={atan2(\y2,\x2)-atan2(\y1,\x1)-360*(atan2(\y2,\x2)-atan2(\y1,\x1)>0)}, x radius=\qoffs, y radius=\qoffs] --
		(21) let \p1=($(21)-#2$), \p2=($(23)-#2$) in
			arc[start angle={atan2(\y1,\x1)}, delta angle={atan2(\y2,\x2)-atan2(\y1,\x1)-360*(atan2(\y2,\x2)-atan2(\y1,\x1)>0)}, x radius=\qoffs, y radius=\qoffs] --
		(32) let \p1=($(32)-#3$), \p2=($(31)-#3$) in
			arc[start angle={atan2(\y1,\x1)}, delta angle={atan2(\y2,\x2)-atan2(\y1,\x1)-360*(atan2(\y2,\x2)-atan2(\y1,\x1)>0)}, x radius=\qoffs, y radius=\qoffs] --
		cycle}

		\ifx\relax#5\relax
		\def\qlwidth{1pt}
	\else
		\def\qlwidth{#5}
	\fi
	
		\ifx\relax#7\relax
		\fill \nqhedge;
	\else
		\fill[#7]\nqhedge;
	\fi

		\ifx\relax#6\relax
		\draw[dotted, line width=\qlwidth,rounded corners=\qoffs]\nqhedge;
	\else
		\draw[dotted, line width=\qlwidth,#6]\nqhedge;
	\fi
}

\let\epsilon=\varepsilon
\let\eps=\epsilon
\let\rho=\varrho
\let\theta=\vartheta

\def\NN{{\mathds N}}

\DeclareMathOperator{\Ends}{ends}

\newcommand{\cC}{\mathcal{C}}

\newcommand{\cP}{\mathcal{P}}

\newcommand{\cS}{\mathcal{S}}

\newcommand{\ccC}{\mathcal{C}}

\newtheoremstyle{note}  {4pt}  {4pt}  {\sl}  {}  {\bfseries}  {.}  {.5em}          {}
\newtheoremstyle{introthms}  {3pt}  {3pt}  {\itshape}  {}  {\bfseries}  {.}  {.5em}          {\thmnote{#3}}
\newtheoremstyle{remark}  {2pt}  {2pt}  {\rm}  {}  {\bfseries}  {.}  {.3em}          {}

\theoremstyle{plain}
\newtheorem{theorem}{Theorem}[section]
\newtheorem{lemma}[theorem]{Lemma}

\newtheorem{cor}[theorem]{Corollary}

\newtheorem{claim}[theorem]{Claim}
\newtheorem{quest}[theorem]{Question}

\theoremstyle{note}
\newtheorem{definition}[theorem]{Definition}

\theoremstyle{remark}
\newtheorem{remark}[theorem]{Remark}

\newtheorem{exmpl}[theorem]{Example}

\usepackage{accents}

\usepackage{lineno}
\newcommand*\patchAmsMathEnvironmentForLineno[1]{
\expandafter\let\csname old#1\expandafter\endcsname\csname #1\endcsname
\expandafter\let\csname old#1\expandafter\endcsname\csname end#1\endcsname
\renewenvironment{#1}
{\linenomath\csname old#1\endcsname}
{\csname oldend#1\endcsname\endlinenomath}}

\AtBeginDocument{
}

\def\outd{d^+\!}
\def\ind{d^-\!}
\newcommand{\NATS}{\mathds{N}}
\usepackage{cleveref}
\newcommand{\expectation}{\mathbf{E}}
\newcommand{\probability}{\mathbf{P}}
\usepackage{stmaryrd}
\newenvironment{proofclaim}[1][Proof of the claim]{\begin{proof}[#1]}{\end{proof}}

\begin{document}

\title[Cycle decompositions in $3$-uniform hypergraphs]{Cycle decompositions in $3$-uniform hypergraphs}
\author[S. Piga]{Sim\'on Piga}
\address{Fachbereich Mathematik, Universit\"{a}t Hamburg, Hamburg, Germany}
\email{simon.piga@uni-hamburg.de}
\author[N. Sanhueza-Matamala]{Nicol\'as Sanhueza-Matamala}
\address{The Czech Academy of Sciences, Institute of Computer Science, Pod Vod\'{a}renskou v\v{e}\v{z}\'{\i} 2, 182 07 Prague, Czechia}
\email{nicolas@sanhueza.net}
\thanks{
The research leading to these results was supported by the Czech Science Foundation, grant number GA19-08740S with institutional support RVO: 67985807 (N.~Sanhueza-Matamala) and by ANID/CONICYT Acuerdo Bilateral DAAD/62170017 through a Ph.D. Scholarship (S.~Piga)}


\begin{abstract}
	We show that $3$-graphs on $n$ vertices whose codegree is at least $(2/3 + o(1))n$ can be decomposed into tight cycles and admit Euler tours, subject to the trivial necessary divisibility conditions.
	We also provide a construction showing that our bounds are best possible up to the $o(1)$ term.
	All together, our results answer in the negative some recent questions of Glock, Joos, Kühn, and Osthus.
\end{abstract}

\maketitle

\section{Introduction}

\subsection{Cycle decompositions}

Given a~$k$-uniform hypergraph~$H$, a~\emph{decomposition of~$H$} is a collection of subgraphs of $H$ such that every edge of~$H$ is covered exactly once.
When these subgraphs are all isomorphic copies of a single hypergraph~$F$ we say that it is an~\emph{$F$-decomposition}, and that $H$ is \emph{$F$-decomposable}.
Finding decompositions of hypergraphs is one of the oldest problems in combinatorics.
For instance, the well-known problem of the existence of designs and Steiner systems can be cast as the problem of decomposing a complete hypergraph into smaller complete hypergraphs of a fixed size.
Thanks to the recent breakthroughs of Keevash~\cite{Keevash2014} and Glock, Kühn, Lo, and Osthus~\cite{GKLO2020} our knowledge about hypergraph decompositions has increased substantially; but many open questions remain.
We refer the reader to the survey of Glock, Kühn, and Osthus~\cite{GlockKuhnOsthus2020} for an overview of the state of the art.

Here we focus in decompositions in which the subgraphs are all cycles.
For $k \ge 2$ and $\ell \geq k+1$, the \textit{$k$-uniform tight cycle of length~$\ell$} is the $k$-graph $C^k_{\ell}$ whose vertices are $\{v_1, v_2, \dots, v_{\ell} \}$ and whose edges are all $k$-sets of consecutive vertices of the form $\{ v_i,v_{i+1}, \dotsc, v_{i+k-1} \}$ for $1 \leq i \leq \ell$, where the indices are understood modulo $\ell$.
Since no other kind of hypergraph cycles will be considered, we will refer to tight cycles as \emph{cycles}.
If $k$ is clear from the context, we will just write $C_{\ell}$ instead of $C^k_{\ell}$.

Given a vertex~$x$ in~$H$ the \emph{degree of~$x$,~$\deg_H(x)$}, is the number of edges that contain~$x$. 
For a positive integer~$k$, when the degree of every vertex of a hypergraph~$H$ is divisible by~$k$ we say that~$H$ is~\emph{$k$-vertex-divisible}.
Note that in a $k$-uniform cycle every vertex has degree exactly $k$.
This implies that, for any $\ell \ge k+1$, any $C^k_{\ell}$-decomposable $k$-graph $H$ must necessarily be $k$-vertex-divisible.
Another obvious necessary condition to find $C_{\ell}$-decompositions in $H$ is that the total number of edges of $H$ must be divisible by $\ell$.
If $H$ satisfies these two conditions, we say that $H$ is \emph{$C_{\ell}$-divisible}.

However, not every $C_{\ell}$-divisible $k$-graph is $C_{\ell}$-decomposable.
For instance, a cycle $C_{2\ell}$ is $C_{\ell}$-divisible, but clearly does not have a $C_{\ell}$-decomposition.
This motivates the search of easily-checkable sufficient conditions which, together with the necessary $C_{\ell}$-divisibility, already force the existence of $C_{\ell}$-decompositions.
A natural choice is to consider degree conditions, which in hypergraphs can be expressed in terms of \emph{codegree}.
For $k$-uniform graphs and a set $S$ of $(k-1)$ vertices,
we define the \emph{codegree of~$S$},~$\deg_H(S)$, as the number of edges of $H$ that contain all of $S$.
We denote the minimum (resp. maximum) codegree of a hypergraph~$H$ over all $S$ by~$\delta_{k-1}(H)$ (resp. $\Delta_{k-1}(H)$).
The \emph{$C^k_{\ell}$-decomposition threshold} $\delta_{\smash{C^k_{\ell}}}(n)$ is the minimum $d$ such that every $C^k_{\ell}$-divisible $k$-graph $H$ on $n$ vertices with $\delta_{k-1}(H) \geq d$ is $C^k_{\ell}$-decomposable.
Moreover, it is convenient to define~$\delta_{\smash{C^k_{\ell}}} = \limsup_{n\rightarrow \infty}\delta_{\smash{C^k_{\ell}}}(n)/n$.
Again, we may omit $k$ from the notation and write $\delta_{\smash{C_{\ell}}}(n)$ and $\delta_{\smash{C_{\ell}}}$.
The very general results of~\cite{GKLO2020} imply that $\delta_{\smash{C^k_{\ell}}} < 1$ for all $k \ge 2$ and $\ell > k$, but no precise values are known when $k \ge 3$.

In our main result, we find the value of $\delta_{\smash{C^3_{\ell}}}$ for all but finitely many values of $\ell$.

\begin{theorem} \label{theorem:ldecomposcodegree}
    Suppose $\ell$ satisfies one of the following:
    \begin{enumerate*}[{\upshape (i)}]
        \item $\ell$ is divisible by $3$ and at least $9$, or
        \item $\ell \ge 10^7$.
    \end{enumerate*}
    Then $\delta_{\smash{C^3_{\ell}}} =2/3$.
\end{theorem}

Theorem~\ref{theorem:ldecomposcodegree} implies an interesting contrast with respect to what is known for $C^2_{\ell}$-decomposition thresholds, which we now recall.
In graphs (i.e. $2$-uniform hypergraphs), the codegree conditions default to conditions on minimum degree.
Barber, Kühn, Lo, and Osthus~\cite{BKLO2016} introduced the technique of \emph{iterative absorption} to study $F$-decompositions in graphs
---this technique is also crucial to our present work, and will be reviewed in detail in Section~\ref{Section:IterativeAbsorption}.
In particular, for cycle decompositions in graphs, their work implies that $\delta_{\smash{C_{\ell}}}(n) \leq \delta_{\smash{C_{\ell}}}^\ast(n)  + o(n)$.
Here, $\delta_{\smash{C_{\ell}}}^\ast(n)$ is the minimum degree which guarantees the existence of `fractional $C_{\ell}$-decompositions' in $n$-vertex graphs.
This notion corresponds to the natural fractional relaxation of decompositions (we define and discuss this in Section~\ref{subsection:wellbehaved}).
Let $\delta_{\smash{C_{\ell}}}^\ast = \limsup_{n \rightarrow \infty} \delta_{\smash{C_{\ell}}}^\ast(n)/n$.

The famous Nash-Williams conjecture~\cite{NashWilliams1970} says that $\delta_{\smash{C_{3}}}(n) \leq 3n/4$.
This is still open, with the current best upper bound given by~$\delta_{\smash{C_{3}}}^\ast \leq d \approx 0.827$ due to Delcourt and Postle~\cite{MR4163104}.
Very recently, Joos and Kühn~\cite{KuhnJoos2021} proved that $\delta_{\smash{C_{\ell}}}^\ast$ tends to $1/2$ whenever $\ell$ goes to infinity.
Together with the best known lower bounds~\cites{BKLO2016,BGKLMO2020}, we now know that for all odd $\ell \ge 3$,
\[ \frac{1}{2} + \frac{1}{2(\ell - 1)} \leq \delta_{\smash{C_{\ell}}} \leq \delta_{\smash{C_{\ell}}}^\ast \leq \frac{1}{2} + O\left( \frac{\log \ell}{\ell} \right).\]
On the other hand, cycles of even length are bipartite, and Glock, Kühn, Lo, Montgomery, and Osthus~\cite{GKLMO2019} were able to characterise the `decomposition thresholds' for all bipartite graphs.
In particular, $\delta_{\smash{C_{4}}} = 2/3$ and $\delta_{\smash{C_{\ell}}} = 1/2$ for all even $\ell \ge 6$.
Remarkably, Taylor \cite{Taylor2019} showed exact results for large $n$, by proving $\delta_{\smash{C_{4}}}(n) = 2n/3 - 1$ and $\delta_{\smash{C_{\ell}}}(n) = n/2$ for all even $\ell \geq 8$.

To summarise, for large $\ell$ the values of $\delta_{\smash{C^2_{\ell}}}$ have a strong dependence on the parity of $\ell$, being $\delta_{\smash{C^2_{\ell}}} > 1/2$ if $\ell$ is odd, and $\delta_{\smash{C^2_{\ell}}} = 1/2$ otherwise.
In contrast, Theorem~\ref{theorem:ldecomposcodegree} implies that for $k = 3$ and large $\ell$ the behaviour is different: $\delta_{\smash{C^3_{\ell}}} = 2/3$ for all $\ell$ sufficiently large, regardless of whether the cycle is tripartite or not.

The following simple corollary can be deduced from our main theorem.
Say a $k$-graph has a \emph{cycle decomposition} if it admits a decomposition into cycles.
That is, there are edge-disjoint cycles ---not necessarily of the same length--- which cover every edge exactly once.
This notion is weaker than that of having a $C_{\ell}$-decomposition for a fixed $\ell$.
It is easy to see that any $3$-graph having a cycle decomposition must be $3$-vertex-divisible.
As a corollary of Theorem~\ref{theorem:ldecomposcodegree}, we obtain an upper bound on the minimum codegree sufficient to force a cycle decomposition.

\begin{cor} \label{theorem:cycledecomposcodegree}
    Any $3$-vertex-divisible $3$-graph $H$ with $\delta_2(H) \geq (2/3 + o(1))|H|$ has a cycle decomposition.
\end{cor}

\subsection{Euler tours}

Our focus in decompositions into cycles is partly motivated by its close connections with the celebrated problem of finding \textit{Euler tours}.
Given a $k$-graph $H$, a \emph{tour} is a sequence of non-necessarily distinct vertices $v_1, \dots, v_{\ell}$ such that, for every $1 \leq i \leq \ell$ the $k$ consecutive vertices~$\{v_i,v_{i+1},\dots, v_{i+k-1}\}$ induce an edge (understanding the indices modulo $\ell$), and moreover all of these edges are distinct.
If a hypergraph~$H$ contains a tour that covers each edge exactly once, we call it~\emph{Euler tour} and we say that~$H$ is \emph{Eulerian}.

Famously, Euler~\cite{Euler1741} proved that every Eulerian graph must be~$2$-vertex-divisible, 
and stated (later proved by Hierholzer and Wiener~\cite{HierholzerWiener1873}) that connected and $2$-vertex-divisible graphs are Eulerian.
Analogously, for $k \geq 3$, it is an easy observation that every Eulerian $k$-graph must be $k$-vertex-divisible.
However, the characterisation of Eulerian $k$-graphs is not as simple as for $k = 2$.
In fact, until recently, it was not even known if complete $k$-vertex-divisible $k$-graphs were Eulerian.
It was conjectured by Chung, Diaconis, and Graham~\cite{ChungDiaconisGraham1992} that indeed that should be the case, at least for sufficiently large complete $k$-graphs.
This was proven to be true by Glock, Joos, Kühn, and Osthus~\cite{GJKO2020}, which deduced this from a more general result which finds Euler tours in $k$-graphs with certain quasirandom conditions (which are satisfied by complete graphs).

From this more general result, they also deduced a `minimum codegree' version of their theorem: there exists~$c>0$ such that any sufficiently large~$3$-vertex-divisible hypergraph~$H$ with~$\delta_2(H)\geq (1-c)\vert H\vert$ is Eulerian.
The constant~$c$ which they obtained is fairly small (by inspecting their proof, we estimate $\log_2 (c)\leq {-10^{12}}$) and therefore improving the minimum codegree condition becomes a natural problem. 
Their proof is based fundamentally on a reduction to the problem of finding a cycle decomposition.
In the same fashion, we can use Theorem~\ref{theorem:ldecomposcodegree} to improve the minimum codegree condition.

\begin{cor} \label{corollary:codegreeuler}
    Any $3$-vertex-divisible $3$-graph $H$ with $\delta_2(H) \geq (2/3 + o(1))|H|$ is Eulerian.
\end{cor}

\subsection{Lower bounds and counterexamples} \label{subsection:counterexample}

Theorem~\ref{theorem:ldecomposcodegree}, Corollary~\ref{theorem:cycledecomposcodegree} and Corollary~\ref{corollary:codegreeuler} hold for $3$-graphs $H$ satisfying $\delta_2(H) \geq (2/3 + o(1))|H|$.
Glock, Kühn, and Osthus~\cite[Conjecture 5.6]{GlockKuhnOsthus2020} conjectured that Corollary~\ref{theorem:cycledecomposcodegree} should hold already for any $H$ with~$\delta_2(H)\ge (1/2+o(1))\vert H \vert$.
Similarly, in the setting of Corollary~\ref{corollary:codegreeuler}, Glock, Joos, Kühn, and Osthus~\cite[Conjecture 3]{GJKO2020} conjectured (reiterated in~\cite[Conjecture 5.4]{GlockKuhnOsthus2020}) that a minimum codegree of~$(1/2+o(1))|H|$ should be enough to guarantee the existence of Euler tours.

However, it turns out that the `$2/3$' in our statements cannot be lowered.
We prove this by constructing a family of counterexamples which are able to cover all of the previous settings ($C_{\ell}$-decompositions, cycle decompositions, and Euler tours) in a unified way.

A \emph{tour decomposition} of $H$ is a collection of edge-disjoint tours in $H$ which, together, cover all edges of $H$.
Note that a cycle is precisely a tour which does not repeat vertices.
Thus we have that both $C_{\ell}$-decompositions and cycle decompositions are particular instances of tour decompositions,
and moreover Eulerian graphs are graphs which admit a tour decomposition consisting of a single tour.
Thus the following result shows that Theorem~\ref{theorem:ldecomposcodegree}, Corollary~\ref{theorem:cycledecomposcodegree}, and Corollary~\ref{corollary:codegreeuler} are asymptotically tight for the minimum codegree condition.

\begin{theorem} \label{theorem:counterexample}
    Let $\ell \ge 4$ and $n \ge 3(\ell+3)$ be divisible by $18$.
    Then there exists a $C_{\ell}$-divisible $3$-graph $H$ on $n$ vertices
    which satisfies $\delta_2(H) \geq (2n - 15)/3$,
    but does not admit a tour decomposition.
\end{theorem}

\subsection{Organisation of the paper}
In Section~\ref{section:counterexample} we prove the lower bound of Theorem~\ref{theorem:counterexample}.
In Section~\ref{section:corollaries} we give short proofs of Corollaries~\ref{theorem:cycledecomposcodegree} and~\ref{corollary:codegreeuler} assuming Theorem~\ref{theorem:ldecomposcodegree}.

In Section~\ref{Section:IterativeAbsorption} we show Theorem~\ref{theorem:ldecomposcodegree} by using the technique of \emph{iterative absorption}, which we review there. The technique relies on three main lemmata, the Vortex Lemma, Cover-Down Lemma and Absorbing Lemma.
After some useful tools (Section~\ref{section:tools}), these three lemmata are proved in Sections~\ref{section:vortex}, \ref{section:cover} and~\ref{section:absorbing}, respectively.
We finish in Section~\ref{section:final} with some remarks and questions.

\subsection{Notation} \label{subsection:notation}

Since isolated vertices make no difference in our context, we usually do not distinguish from a hypergraph~$H=(V,E)$ and its set of edges~$E$.
We will suppress brackets and commas to refer to pairs and triples of vertices when they are considered as edges of a hypergraph. 
For instance, for $x, y, z \in V(H)$, $xyz \in H$ means that the edge $\{x, y, z\}$ is in $E(H)$.
For a vertex~$x\in V(H)$, the \emph{link graph of~$x$} is the~$2$-graph $H(x)$ with edge set~$\{yz\in \binom{V}{2}\colon xyz\in E(H)\}$.
Moreover, given a set of vertices~$U\subseteq V$ we denote the~\emph{restricted link graph} by~$H(v,U) = H(v)\cap \binom{U}{2}$. 
The~\emph{degrees}~$\deg_H(x)$ and~$\deg_H(x,U)$ correspond to~$\vert H(x)\vert$ and~$\vert H(x,U)\vert$ respectively. 
For a pair of vertices $xy$ in $V(H)$, the \emph{neighbourhood of~$xy$} $N_H(xy)$ is the set of vertices $z \in V(H)$ such that $xyz \in H$, given $U \subseteq V(H)$ then $N_H(xy, U) = N_H(xy) \cap U$.
The \emph{codegrees} $\deg_H(xy)$ and $\deg(xy, U)$ correspond to $|N_H(xy)|$ and $|N_H(xy, U)|$ respectively.
We suppress $H$ from the degrees, codegrees, and neighbourhoods if it can be deduced from context.
The \emph{shadow} $\partial H$ of a $3$-graph $H$ is $\{ uv \in \binom{V(H)}{2} : \deg(uv) > 0 \}$.
If $\mathcal{C} = \{ C_1, \dotsc, C_{r} \}$ is a collection of subgraphs of $H$,
sometimes we will let $E(\cC)$ be the hypergraph whose edges are $\bigcup_{1 \leq i \leq r} E(C_i)$.

We will use hierarchies in our statements.
The phrase ``$a \ll b$'' means ``for every $b > 0$, there exists $a_0 > 0$, such that for all $0 < a \leq a_0$ the following statements hold''.
We implicitly assume all constants in such hierarchies are positive, and if $1/a$ appears we assume $a$ is an integer.

A \emph{walk} in a $3$-graph $H$ is a sequence $W = (v_1, \dotsc, v_{\ell})$ of vertices of $H$ such that every $3$ consecutive vertices form an edge of $H$.
A \emph{trail} is a walk in which no edge appears more than once, and a \emph{path} is a trail in which no vertex appears more than once.
A \emph{closed walk} is a walk in which every cyclic shift is still a walk of $H$ (thus tours are trails which are closed walks).
Given a walk $W = (v_1, v_2, \dotsc, v_{\ell})$, we define its \emph{start} ${s}(W)$ and
\emph{terminus} $t(P)$ as $\{v_1, v_2\}$ and $\{v_{\ell - 1}, v_{\ell}\}$ respectively, and we say $W$ \emph{goes from} $(v_1,v_2)$ to $(v_{\ell-1}, v_\ell)$ and also that $W$ is a \emph{$(v_1, v_2, v_{\ell - 1}, v_{\ell})$-path}.
We will use the simpler notation $W = v_1 v_2 \dotsb v_{\ell}$ for walks, and, when useful, we will identify such walks with subgraphs of $H$ (so we can say e.g. $e \in E(W)$).

\section{Lower bounds}\label{section:counterexample}

In this section we prove Theorem~\ref{theorem:counterexample}.
The following lemma captures divisibility constraints that tours in $3$-graphs must satisfy, and it will be the basis of our constructions.
For a $3$-graph $H$, a subgraph $W \subseteq H$ and vertex sets $X,Y,Z$ in $V(H)$, let $W[X,Y,Z]$ be the set of edges $xyz$ in $E(W)$ such that~$x\in X$,~$y\in Y$, and~$z\in Z$.

\begin{lemma} \label{lemma:crucial}
    Let $H$ be a $3$-graph with a vertex partition $\{U_0, U_1, U_2 \}$, and $H[U_0, U_1, U_2] = \emptyset$.
    Let $W$ be a tour in $H$.
    Then $|W[U_1, U_1, U_2]| \equiv |W[U_1, U_2, U_2]| \bmod 3$.
\end{lemma}

\begin{proof}
	Let $W = w_1 w_2 \dotsb w_{r}$, in cyclic order,
	and let $P = \sigma_1 \dotsb \sigma_{r}$ be a cyclic word over the symbols $\{0,1,2\}$, where $\sigma_i = j$ if and only if $w_i \in U_j$.
	Since $W$ is a tour, it does not repeat edges.
	Thus we have that $|W[U_1, U_1, U_2]|$ is exactly the same as the number of appearances of the patterns $F_1 = \{ 112, 121, 211 \}$ formed by three consecutive symbols in $P$.
	Similarly, $|W[U_1, U_2, U_2]|$ is exactly counted by the number of appearances of $F_2 = \{ 122, 212, 221 \}$ consecutively in $P$.
	In both cases we count the cyclic appearances of the patterns, i.e. we also consider the patterns formed by $\sigma_{r-1} \sigma_{r} \sigma_1$ and $\sigma_{r} \sigma_{1} \sigma_2$.
	
	Define $\Phi(P)$ as follows.
	Scan the triples of consecutive symbols of $P$ one by one, and if they belong to $F_1 \cup F_2$, we add the sum of the values of their symbols to $\Phi(P)$.
	More formally, let $I \subseteq [r]$ be such that $i \in I$ if and only if $\sigma_i \sigma_{i+1} \sigma_{i+2} \in F_1 \cup F_2$ (where the indices are always understood modulo $r$, i.e. $\sigma_{r+1} = \sigma_1$ and $\sigma_{r+2} = \sigma_2$),
	and then 
	\[ \Phi(P) = \sum_{i \in I} ( \sigma_i + \sigma_{i+1} + \sigma_{i+2} ).\]
	
	We aim to show that $\Phi(P) \equiv 0 \bmod 3$.
	If $I = \emptyset$, this is obvious, and if $I = [r]$ then $\Phi(P)$ sums every symbol of $P$ three times, and thus also $\Phi(P) \equiv 0$.
	Thus we can assume $I \notin \{ \emptyset, [r] \}$.
	We write $I$ as a disjoint union of intervals of consecutive indices, minimising the number of intervals.
	Thus, without loss of generality (after shifting $W$ and $P$ cyclically) we can assume $I = I_1 \cup \dotsb \cup I_k$, so each $I_j$ is of the form $\{a_j, a_j + 1 ,\dotsc, b_j \}$ for some $a_j \leq b_j$ and further we have $a_1 = 1$, $b_j \leq a_{j+1} - 2$ for all $1 \leq j < k$ and $b_k \leq r-1$.
	Setting $\Phi_j = \sum_{i \in I_j} (\sigma_i + \sigma_{i+1} + \sigma_{i+2})$ we have $\Phi(P) = \sum_{1 \leq j \leq k} \Phi_j$,
	so it is enough to show that $\Phi_j \equiv 0 \bmod 3$ for each $j$.
	
	Let $1 \leq j \leq k$ be arbitrary, for brevity write $a = a_j$ and $b = b_j$.
	Let $P_j = \sigma_a \sigma_{a+1} \dotsm \sigma_{b+1} \sigma_{b+2}$.
	We claim that $P_j$ begins with two repeated symbols.
	Since $I_k \subseteq I$, we have $\sigma_{a} \sigma_{a + 1} \sigma_{a + 2} \in F_1 \cup F_2$, thus in particular $\sigma_{a}$ and $ \sigma_{a + 1}$ must be in $\{1,2\}$.
	If $\sigma_a \neq \sigma_{a+1}$, then we would have $\sigma_{a} \sigma_{a + 1} = 12$ or $\sigma_{a} \sigma_{a + 1} = 21$.
	In any case, it cannot happen that $\sigma_{a - 1} \in \{1,2\}$, since then that would imply that $a - 1 \in I$, contradicting the choice of $I_k$.
	Thus $\sigma_{a - 1} = 0$, and therefore $\sigma_{a - 1} \sigma_{a} \sigma_{a+1} = 012$ or $\sigma_{a - 1} \sigma_{a} \sigma_{a+1} = 021$.
	But this implies that $W$ contains an edge in $H[U_0, U_1, U_2]$, a contradiction.
	Thus $P_j$ begins with two repeated symbols, and an analogous argument implies that $P_j$ also ends with two repeated symbols.
	
	If $a = b$, then we would have $\sigma_{a} \sigma_{a+1} \sigma_{a+2} = 111$ or $\sigma_{a} \sigma_{a+1} \sigma_{a+2} = 222$, then implying $a \notin I$, a contradiction.
	Thus $a < b$, and therefore $P_j$ must have the form $P_j = xxQ_jyy$, where $x, y \in \{1,2\}$ and $Q_j$ is a (possibly empty) word.
	Thus we have
	\[ \Phi_j = \sum_{a \leq i \leq b} (\sigma_i + \sigma_{i+1} + \sigma_{i+2}) = x + 2x + 3 \left( \sum_{a+2 \leq i \leq b} \sigma_i \right) + 2y + y \equiv 0 \bmod 3, \] 
	and this implies $\Phi(P) \equiv 0 \bmod 3$, as discussed before.
	
	Finally, note that, for $j \in \{1,2\}$, if $\sigma_i \sigma_{i+1} \sigma_{i+2} \in F_j$, then $\sigma_i + \sigma_{i+1} + \sigma_{i+2} \equiv j \bmod 3$.
	Thus $\Phi(P) \equiv |W[U_1, U_1, U_2]| + 2 |W[U_1, U_2, U_2]| \bmod 3$.
	But since $\Phi(P) \equiv 0 \bmod 3$ and $2 \equiv -1 \bmod 3$, we deduce $|W[U_1, U_1, U_2]| \equiv |W[U_1, U_2, U_2]| \bmod 3$, as desired.
\end{proof}

To prove Theorem~\ref{theorem:counterexample},
we will consider alterations of the following $3$-graph.

\begin{definition} \label{definition:H}
	Let $n$ be divisible by $18$ and write $n = 18k$.
	Consider the $3$-graph $H_{n}$ on $n$ vertices, whose vertex set is partitioned into three clusters $V_0, V_1, V_2$ whose sizes are $n_0, n_1, n_2$ respectively, and are defined by
	\begin{align}
		n_0 = 6k, \qquad
		n_1 = 6k-2, \qquad
		\text{and} \qquad
		n_2 = 6k+2. \label{equation:clustersizes}
	\end{align}
	Given a vertex $x \in V(H_{n})$, the \emph{label} $l(x)$ of $x$ is $i$ if and only if $x \in V_i$.
	The edge set of $H_{n}$ is 
	$$E(H_{n})=\{ xyz : l(x)+l(y)+l(z) \not\equiv 0 \bmod 3 \}.$$
\end{definition}

In words, every $3$-set is present as an edge in $H_{n}$, except for those which are entirely contained in one of the clusters $V_i$, or have non-empty intersection with all three clusters.
Usually $n$ will always be clear from context, and for a cleaner notation we will just write $H = H_n$ in the remainder of this section.

We begin our analysis by noting the $3$-graph $H$ has large minimum codegree.

\begin{lemma} \label{lemma:Hcodegree}
	Let $n \in 18 \mathbb{N}$. 
	Then $\delta_2(H) \geq (2n-12)/3$.
\end{lemma}

\begin{proof}
	Let $x,y\in V(H)$, and set $p = l(x)+l(y)$.
	By the definition of $H$, a vertex $z$ will form an edge together with $xy$ whenever $p+l(z) \not\equiv 0 \bmod 3$.
	This is equivalent to $l(z) \equiv 1 - p \bmod 3$ or $l(z) \equiv 2 - p \bmod 3$.
	Thus, if $i, j \in \{0, 1, 2\}$ are such that $i \equiv 1 - p \bmod 3$ and $j \equiv 2 - p \bmod 3$, then $N(xy) = (V_i \cup V_j) \setminus \{x, y\}$.
	A quick case analysis reveals that $|N(xy)|$ is minimised whenever $x \in V_0$, $y \in V_1$, and in such a case $\deg_{H}(xy) = n_0+n_1-2 = 12k-4$.
	Thus $\delta_2(H) = 12k-4 = (2n-12)/3$, as required.
\end{proof}

We note that equations~\eqref{equation:clustersizes} imply that, for $n = 18k$, all $n_0, n_1, n_2$ are even, and for all $i \in \{0,1,2\}$ we have
\begin{align}
	n_i & \equiv i \pmod 3, \label{equation:clustermodulo}
\end{align}

Given $(i, j, k) \in \{0, 1, 2\}^3$, write $H_{ijk} = H[V_i, V_j, V_k]$.

\begin{lemma} \label{lemma:Hmod3}
	Let $n \in 18 \mathbb{N}$.
	Then
	\begin{enumerate}[{\upshape(M1)}]
		\item \label{item:degreeH} for every $x \in V(H)$, $\deg_H(x) \equiv 1 \bmod 3$ and
		\item \label{item:Hmod} $|H_{112}| \not\equiv |H_{122}| \bmod 3$.
	\end{enumerate}
\end{lemma}

\begin{proof}
	We begin by noting that $\binom{m}{2} \equiv 2m(m-1) \bmod 3$ holds for all integers $m$.
	Thus $\binom{m}{2} \equiv 1 \bmod 3$ if $m \equiv 2 \bmod 3$, and $\binom{m}{2} \equiv 0 \bmod 3$ otherwise.
	
	Now let $x \in V_0$.
	Then the pairs $yz$ such that $xyz \in H$ are those such that
	\begin{enumerate}
		\item $y \in V_0 \setminus \{ x\}$ and $z \in V_1 \cup V_2$, of which there are $(n_0 - 1)(n_1 + n_2)$ many,
		\item $yz \subseteq V_1$, of which there are $\binom{n_1}{2}$ many, and
		\item $yz \subseteq V_2$, of which there are $\binom{n_2}{2}$ many.
	\end{enumerate}
	Thus we have $\deg_H(x) = (n_0 - 1)(n_1 + n_2)+\binom{n_1}{2}+\binom{n_2}{2}$.
	Together with~\eqref{equation:clustermodulo}, we have that $\deg_H(x) \equiv 0+0+1 \equiv 1 \bmod 3$.
	Analogous calculations show that
	\begin{align*}
		\deg_H(y) 
		&\equiv 0 + 0 + 1
		\equiv 1 \bmod 3 \text{ for }y\in V_1 \text{ and }\\
		\deg_H(z)
		&\equiv 1 + 0 + 0 
		\equiv 1 \bmod 3 \text{ for }z\in V_2,
	\end{align*}
	thus~\ref{item:degreeH} holds.

	Finally, the sizes of $|H_{112}|$ and $|H_{122}|$ are $\binom{n_1}{2} n_2$ and $\binom{n_2}{2}n_1$ respectively, which then are easily seen to be equivalent to $0$ and $1$ modulo 3, respectively, which implies \ref{item:Hmod}.
\end{proof}

Since $H$ is not quite $3$-vertex-divisible,
our counterexample will consist actually of a slight alteration of $H$ obtained by removing some sparse subgraph,
which we define now.

\begin{lemma} \label{lemma:F}
	Let $n \in 18 \mathbb{N}$.
	Then there exists a perfect matching $F \subseteq H \setminus (H_{112} \cup H_{122})$.
\end{lemma}

\begin{proof}
	Let~$k$ be such that~$n=18k$.
	Let $a, b$ be two distinct vertices in $V_2$,
	and let $V'_1 = V_1 \cup \{a, b\}$ and $V'_2 = V_2 \setminus \{ a, b \}$.
	Note that $|V_0| = |V'_1| = |V'_2| = 6k$.
	Let $V_0 = \{ x_1, \dotsc, x_{6k} \}$, $V'_1 = \{ y_1, \dotsc, y_{6k} \}$ and $V'_2 = \{ z_1, \dotsc, z_{6k} \}$, with $y_1 = a$ and $y_2 = b$.
	Then 
	$$F = \{ y_{2i-1} y_{2i} x_{2i-1} : 1 \leq i \leq 3k \} \cup \{ z_{2i-1} z_{2i} x_{2i} : 1 \leq i \leq 3k \}$$
	is a perfect matching in which %
	every edge intersects $V_0$ in exactly one vertex.
	Thus $F$ has no edge in $H_{112} \cup H_{122}$, as required.
\end{proof}

We are now ready to show Theorem~\ref{theorem:counterexample}.

\begin{proof}[Proof of Theorem~\ref{theorem:counterexample}]
	Consider the $3$-graph $H = H_n$ given in Definition~\ref{definition:H},
	and consider the perfect matching $F \subseteq H \setminus (H_{112} \cup H_{122})$ given by Lemma~\ref{lemma:F}.
	Let $\ell' \in \{ 4, \dotsc, \ell + 3 \}$ be such that $|E(H-F)| + \ell' \equiv 0 \bmod \ell$.
	Since $n = 18k \ge 3(\ell + 3)$, we have $|V_0| = 6k \ge \ell+3 \ge \ell'$.
	To $H-F$, we add a cycle $C$ of length $\ell'$, edge-disjoint from $H-F$, which is entirely contained in $V_0$.
	We claim $H' = (H \setminus F) \cup C$ has all of the desired properties.
	
	We first check $H'$ is $C_{\ell}$-divisible.
	We start by checking $H'$ is $3$-vertex-divisible.
	Indeed, let $x \in V(H')$ be arbitrary.
	We have $\deg_H(x) \equiv 1 \bmod 3$ by Lemma~\ref{lemma:Hmod3}\ref{item:degreeH},
	we have $\deg_F(x) = 1$ since $F$ is a perfect matching,
	and $\deg_C(x) \equiv 0 \bmod 3$ since $C$ is a cycle on $\ell' \ge 4$ vertices.
	Thus $\deg_{H'}(x) \equiv 1-1+0 \equiv 0 \bmod 3$ for all $x \in V(H')$, as required.
	Also, the number of edges of $H'$ is $|E(H')| = |E(H-F)| + \ell'$, which was chosen to be divisible by $\ell$, so indeed $H'$ is $C_{\ell}$-divisible.
	
	Now we check $H'$ has large codegree.
	It suffices to show $H-F$ has large codegree.
	Removing a perfect matching from $H$ decreases the codegree of every pair at most by $1$,
	thus by Lemma~\ref{lemma:Hcodegree}, we have $\delta_2(H-F) \geq \delta_2(H) - 1 \geq (2n-12)/3 - 1 = (2n - 15)/3$.
	
	Now we prove $H'$ does not have a tour decomposition.
	First, since $F \subseteq H \setminus (H_{112} \cup H_{122})$,
	we have $H'[V_1, V_1, V_2] = H_{112}$ and $H'[V_1, V_2, V_2] = H_{122}$.
	For a contradiction, suppose that $W^{1}, \dotsc, W^r$ are tours forming a tour decomposition in $H'$.
	For a walk $W$, let $W_{112} = H_{112} \cap E(W)$,
	and let $W_{122} = H_{122} \cap E(W)$.
	Since the tours are edge-disjoint and cover all edges of $H'$, we have $\sum_{1 \leq i \leq r}|W^{i}_{112}| = |H_{112}|$ and $\sum_{1 \leq i \leq r}|W^{i}_{122}| = |H_{122}|$.
	Since $H_{012} = \emptyset$, Lemma~\ref{lemma:crucial} implies that $|W^i_{112}| \equiv |W^i_{122}| \bmod 3$ for each $1 \leq i \leq r$.
	We deduce $|H_{112}| \equiv |H_{122}| \bmod 3$, but this contradicts Lemma~\ref{lemma:Hmod3}\ref{item:Hmod}.
\end{proof}

\begin{remark}
    For sufficiently large values of $n$, we can make our example vertex-regular instead of $C_{\ell}$-divisible.
    This is needed, for instance, when we are looking at decompositions into spanning vertex-disjoint collections of cycles, such as Hamilton cycles.
    
    Start from $H = H_n$, and remove $F$ as before to get to $H' = H-F$ which is $3$-vertex-divisible.
    Every vertex in $V_i$ has the same degree $d_i$, for all $i \in \{0,1,2\}$, and a calculation reveals that $d_1 = d_0 - 9$ and $d_2 = d_0 - 3$.
    Then, adding $3$ edge-disjoint Hamilton cycles to $H[V_1]$ and one Hamilton cycle to $H[V_2]$ leaves a $3$-graph $H^\ast$ in which every vertex has degree $d_0$, and it can be similarly proved that $H^\ast$ does not admit any tour decomposition.
\end{remark}

\section{Proof of Corollaries~\ref{theorem:cycledecomposcodegree} and and~\ref{corollary:codegreeuler}} \label{section:corollaries}

In this short section we deduce Corollaries~\ref{theorem:cycledecomposcodegree}
and~\ref{corollary:codegreeuler} from Theorem~\ref{theorem:ldecomposcodegree}.

\begin{proof}[Proof of Corollary~\ref{theorem:cycledecomposcodegree}]
	Let $m$ be the number of edges of $H$, and write it as $m = 9q+r$ for some $q \ge 1$ and $0 \leq r < 9$.
	Find a cycle $C$ of length $9+r$ in $H$: this can be done greedily (see Section~\ref{subsection:paths} for details).
	Then, $H' = H - C$ is a $3$-divisible graph, its minimum codegree is $\delta_2(H') \ge \delta_2(H) - 2 \ge (2/3 + \eps/2)n$, and its number of edges is $m-(9+r) = 9(q-2)$, which is divisible by $9$.
	By Theorem~\ref{theorem:ldecomposcodegree}, $H'$ has a $C_{9}$-decomposition, together with $C$ this is a cycle decomposition of~$H$.
\end{proof}

For the proof of Corollary~\ref{corollary:codegreeuler} we use the strategy of Glock, Joos, Kühn, and Osthus~\cite{GJKO2020}.
Crucial part of their argument is (using our terminology) to first find a trail $W$ which is \emph{spanning} (i.e. every $2$-tuple of distinct vertices of $H$ is contained as a sequence of consecutive vertices of $W$) but at the same time is sparse (it satisfies $\Delta_{2}(W) = o(n)$).

We state their relevant lemma only in the particular case $k = 3$.
A $3$-graph $H$ on $n$ vertices is \emph{$\alpha$-connected} if for all distinct $v_1, v_2, v_4, v_5 \in V(H)$, there exist at least $\alpha n$ vertices $v_3 \in V(H)$ such that $v_1 v_2 v_3 v_4 v_5$ is a walk in $H$.

\begin{lemma}[{\cite[Lemma 5]{GJKO2020}}] \label{lemma:randomwalk}
	Suppose $n\in \NN$ is sufficiently large in terms of $\alpha$.
	Suppose $H$ is an $\alpha$-connected $3$-graph on $n$ vertices.
	Then $H$ contains a spanning trail $W$ satisfying $\Delta_2(W) \leq \log^3 n$.
\end{lemma}

\begin{proof}[Proof of Corollary~\ref{corollary:codegreeuler}]
	Take $n_0$ such that $1/n_0 \ll \eps$.
	Since $H$ satisfies $\delta_2(H) \ge (2/3 + \eps)n$, it is $\eps$-connected.
	By Lemma~\ref{lemma:randomwalk} there exists a spanning trail $W = w_1 \dotsb w_r$ satisfying $\Delta_2(W) \leq \log^3 n$.
	Use the $\eps$-connected property of $H$ to close $W$ to a tour, using three extra vertices, while avoiding edges previously used by $W$ (using that $\Delta_2(W) \leq \log^3 n$).
	The resulting $W' = w_1 \dotsb w_{r+3}$ is a spanning tour which satisfies $\Delta_2(W') \leq 2 \log^3 n$.
	Let $H' = H - W'$.
	Since $W'$ is a tour and $H$ is $3$-vertex-divisible, $W'$ is $3$-vertex-divisible as well.
	Since $\Delta_2(W') \leq 2 \log^3 n \leq \eps n / 2$ and $\delta_2(H) \ge (2/3 + \eps)n$,
	we deduce $\delta_2(H') \ge (2/3 + \eps/2)n$.
	Since $n$ is sufficiently large,
	Corollary~\ref{theorem:cycledecomposcodegree} implies that $H'$ has a cycle decomposition.
	Fix one of those cycles $C= v_1 v_2 \dotsb v_m$ and note that the ordered pair $(v_1, v_2)$ must appear consecutively in some part of $W'$ (since $W'$ is spanning).
	We may write $W' = W'_1 v_1 v_2 W'_2$ and extend $W'$ by taking $W'_1 v_1 v_2 \dotsb v_m v_1 v_2 W'_2$, which is still an spanning tour, but now uses the edges of $C$ in addition to those of $W'$.
	Attaching the cycles of the decomposition  one by one to $W'$, we obtain the desired Euler tour.
\end{proof}

\section{Iterative absorption: proof of Theorem~\ref{theorem:ldecomposcodegree}}\label{Section:IterativeAbsorption}

Our proof of Theorem~\ref{theorem:ldecomposcodegree} follows the strategy of \emph{iterative absorption} introduced by Barber, Kühn, Lo, and Osthus~\cite{BKLO2016} and further developed by Glock, Kühn, Lo, Montgomery, and Osthus \cite{GKLMO2019} to study decomposition thresholds in graphs.
We base our outline in the exposition of Barber, Glock, Kühn, Lo, Montgomery, and Osthus~\cite{BGKLMO2020}.

The method of iterative absorption rests around three main lemmata,
originally called the the Vortex Lemma, Absorbing Lemma, and the Cover-Down Lemma.
We will introduce these lemmata first while explaining the global strategy,
then we will use them to prove Theorem~\ref{theorem:ldecomposcodegree}.
The proof of these lemmata will take up the rest of the paper.

A sequence of nested subsets of vertices~$U_0\supseteq U_1\supseteq\dots \supseteq U_\ell$ is called a \emph{$(\delta, \xi, m)$-vortex in~$H$} if satisfies the following properties. 

\begin{enumerate}[(V1)]
	\item $U_0 = V(H)$,
	\item for each $1 \leq i \leq \ell$, $|U_i| = \lfloor \xi | U_{i-1}| \rfloor$,
	\item $|U_{\ell}| = m$, and
	\item $\deg(x, U_i) \ge \delta \binom{|U_i|}{2}$ for each $1 \leq i \leq \ell$ and $x \in U_{i-1}$, and
	\item $\deg(xy, U_i) \ge \delta |U_i|$ for each $1 \leq i \leq \ell$ and $xy \in \binom{U_{i-1}}{2}$.
\end{enumerate}

The existence of vortices for suitable parameters~$\delta$,~$\xi$, and~$m$ is stated in the Vortex Lemma. 

\begin{lemma}[Vortex Lemma] \label{lemma:vortexlemma}
	Let $\xi, \delta > 0$ and $m' \in \NATS$ be such that $1/m' \ll \xi$.
	Let $H$ be a $3$-graph on $n \ge m'$ vertices with $\delta_2(H) \ge \delta$.
	Then it has a $(\delta - \xi, \xi, m)$-vortex, for some $\lfloor \xi m' \rfloor \leq m \leq m'$.
\end{lemma}

The main idea is to use the properties of the vortex to find a suitable~\emph{$C_\ell$-packing}, i.e. a collection of edge-disjoint~$C_\ell\subseteq H$.
We will find a packing covering most edges of $H$, and moreover the non-covered edges will lie entirely in~$U_\ell$. 
The Absorbing Lemma will provide us with a small structure that we put aside at the beginning, and that will be used to deal with the small remainder left by our $C_{\ell}$-packing.
If $R \subseteq H$ is a subgraph of $H$, a \emph{$C_{\ell}$-absorber for $R$} is a subgraph $A \subseteq H$, edge-disjoint from $R$, such that both $A$ and $A \cup R$ are $C_{\ell}$-decomposable.

\begin{lemma}[Absorbing Lemma] \label{lemma:absorbinglemma}
	Let $\ell \ge 7$, $\eps > 0$, and $n, m \in \NATS$ such that $1/n \ll \eps, 1/m, 1/\ell$.
	Let $H$ be a $3$-graph on $n$ vertices with $\delta_2(H) \ge (2/3 + \eps)n$.
	Let $R \subseteq H$ be $C_{\ell}$-divisible on at most $m$ vertices.
	Then there exists a $C_{\ell}$-absorber for $R$ in $H$ with at most~$(4m\ell)^9$ edges.
\end{lemma}

Finally, we construct the desired~$C_\ell$-packing step by step through the nested sets of the vortex. 
More precisely, suppose $U_i \supseteq U_{i+1}$ are two consecutive sets in a vortex of $H$.
The Cover-Down Lemma will be applied to find a $C_{\ell}$-packing which covers every edge of $H[U_i]$, except maybe for some in $H[U_{i+1}]$.
Thus the packing will be found via reiterated applications.

\begin{lemma}[Cover-Down Lemma] \label{lemma:coverdownlemma}
	Let $\ell \ge 9$ be divisible by $3$ or at least $10^7$,
	and $\eps, \mu > 0$ and $n \in \NATS$ with $1/n \ll \mu, \eps \ll 1/\ell$.
	Suppose $H$ is a $3$-graph on $n$ vertices, and $U \subseteq V(H)$ with $|U| = \lfloor \eps n \rfloor$, which satisfy
	\begin{enumerate}[{\upshape (C1)}]
	    \item $\delta_2(H) \ge (2/3 + 2\eps)n$,
	    \item $\deg_H(x, U) \ge (2/3 + \eps)\binom{|U|}{2}$ for each $x \in V(H)$,
	    \item $\deg_H(xy, U) \ge (2/3 + \eps)|U|$ for each $xy \in \binom{V(H)}{2}$, and
	    \item $\deg_H(x)$ is divisible by $3$ for each $x \in V(H) \setminus U$.
	\end{enumerate}
	Then $H$ has a $C_{\ell}$-decomposable subgraph $F$ such that $H - H[U] \subseteq F$,
	and $\Delta_2(F[U]) \leq \mu n$.
\end{lemma}

Assuming lemmata~\ref{lemma:absorbinglemma}--\ref{lemma:coverdownlemma}, we prove Theorem~\ref{theorem:ldecomposcodegree} holds
(cf.~\cite[Section 3.4]{BGKLMO2020}).

\begin{proof}[Proof of Theorem~\ref{theorem:ldecomposcodegree}]
	It is enough to show that, for every $\eps > 0$, there exists $n_0$ such that every $C_{\ell}$-divisible $3$-graph $H$ on $n \ge n_0$ vertices with $\delta_2(H) \ge (2/3 + 8 \eps)n$ admits a $C_{\ell}$-decomposition.
	Given $\eps$ and $\ell$, we fix $m',n_0$ such that 
	\begin{align}\label{constants}
	    1/n_0 \ll 1/m' \ll \eps, 1/\ell.
	\end{align}
	Let $H$ on $n \ge n_0$ vertices as before, we are done if we show $H$ has a $C_{\ell}$-decomposition.
	
	\medskip
	\noindent \emph{Step 1: Setting the vortex and absorbers.}
	By Lemma~\ref{lemma:vortexlemma}, $H$ has a $(2/3 + 7\eps, \eps, m)$-vortex $U_0 \supseteq \dotsb \supseteq U_{\ell}$, for some $m$ such that $\lfloor \eps m' \rfloor \leq m \leq m'$.
	
	Let $\mathscr{L}$ be the family of all $C_{\ell}$-divisible $3$-graphs which are subgraphs of $H[U_{\ell}]$.
	Since $|U_{\ell}| = m$, clearly~$\vert \mathscr L\vert \leq 2^{\binom{m}{3}}$.
	Let $L \in \mathscr{L}$ be arbitrary.
	Since~$m\leq m'$ and~\eqref{constants}, a suitable application of Lemma~\ref{lemma:absorbinglemma} yields a~$C_\ell$-absorber~$A_{L}\subseteq H\setminus H[U_1]$ of~$L$ with at most~$(4m\ell)^9$ edges.
	Since $1/n \ll 1/m, \eps, 1/\ell$, removing the edges of $A_L$ only barely affects the codegree of $H$, thus we can repeat the argument to obtain an absorber $A_{L'} \subseteq H\setminus H[U_1]$ for some $L' \neq L$, edge-disjoint from $A_L$.
	Since the total number of $L \in \mathscr{L}$ is tiny with respect to $n$, we can iterate this argument to obtain edge-disjoint~$C_\ell$-absorbers~$A_{L}\subseteq H\setminus H[U_1]$, one for each~$L\in \mathscr L$. 
	Moreover, each~$A_{L}$ contains at most~$(4m\ell)^9$ edges, and hence, the union~$A=\bigcup_{L\in \mathscr L} A_{L} \subseteq H \setminus H[U_1]$ contains at most~$|\mathscr{L}|(4m\ell)^9 \leq 2^{\binom{m}{3}} (4m\ell)^9 \leq \eps n$ edges.
	By construction, we have $A$ is $C_{\ell}$-decomposable and for each $L \in \mathscr{L}$, $L \cup A$ is $C_{\ell}$-decomposable.
	
	Let $H'=H\setminus A$ and observe that~$\delta_2(H') \geq (2/3+7\eps)n$ and~$U_0 \supseteq \dotsb \supseteq U_{\ell}$ is a $(2/3+6 \eps, \eps, m)$-vortex for~$H'$
	(for this, it is crucial that $A \subseteq H \setminus H[U_1]$).
	Notice that since~$A$ and~$H$ are~$C_\ell$-divisible, we get that~$H'$ is~$C_\ell$-divisible.
	
	\medskip
	\noindent \emph{Step 2: The cover-down.}
	Now we aim to find a $C_{\ell}$-packing in $H'$ using every edge of $H' \setminus H'[U_{\ell}]$.
	Let $U_{\ell + 1} = \emptyset$.
	For each $0 \leq i \leq \ell$ we wish to find $H_i \subseteq H'[U_i]$ such that
	\begin{enumerate}[(a$_{i}$)]
	    \item \label{item:decomposition-coverdown-a} $H' - H_i$ has a $C_{\ell}$-decomposition,
	    \item \label{item:decomposition-coverdown-b} $\delta_2(H_i) \ge (2/3 + 4 \eps) |U_i|$,
	    \item \label{item:decomposition-coverdown-c} $\deg_{H_i}(x, U_{i+1}) \ge (2/3 + 5 \eps ) \binom{|U_{i+1}|}{2}$ for all $x \in U_i$,
	    \item \label{item:decomposition-coverdown-d} $\deg_{H_i}(xy, U_{i+1}) \ge (2/3 + 5 \eps ) |U_{i+1}|$ for all $x, y \in U_i$, and
	    \item \label{item:decomposition-coverdown-e} $H_i[U_{i+1}] = H'[U_{i+1}]$.
	\end{enumerate}
	For $i = 0$ this can be done by setting $H_0 = H'$.
	Now suppose $H_i$ satisfying \ref{item:decomposition-coverdown-a}--\ref{item:decomposition-coverdown-e} is given for some $0 \leq i < \ell$, we wish to construct $H_{i+1}$ satisfying \hyperref[item:decomposition-coverdown-a]{(a$_{i+1}$)}--\hyperref[item:decomposition-coverdown-e]{(e$_{i+1}$)}.
	By~\ref{item:decomposition-coverdown-a}, $H_i$ is $C_{\ell}$-divisible.
	Let $H'_{i} = H_i \setminus H_i[U_{i+2}]$.
	By \ref{item:decomposition-coverdown-b}--\ref{item:decomposition-coverdown-d} and $|U_{i+2}| \leq \eps |U_{i+1}| \leq \eps^2 |U_i|$, we have
	\begin{enumerate}[(C1)]
	    \item $\delta_2(H'_i) \ge \delta_2(H_i) - |U_{i+2}| \ge (2/3 + 3 \eps)|U_i|$,
	    \item $\deg_{H'_i}(x, U_{i+1}) \ge \deg_{H_i}(x, U_{i+1}) - |U_{i+2}|(|U_{i+1}|-1) \ge (2/3 + 3 \eps) \binom{|U_{i+1}}{2}$, for each $x \in U_i$,
	    \item $\deg_{H'_i}(xy, U_{i+1}) \ge \deg_{H'_i}(xy, U_{i+1}) - |U_{i+2}| \ge (2/3 + 4 \eps) |U_{i+1}|$ for each $x, y \in U_i$, and
	    \item $\deg_{H'_i}(x)$ is divisible by $3$ for each $x \in U_i \setminus U_{i+1}$.
	\end{enumerate}
	This allows us to apply Lemma~\ref{lemma:coverdownlemma} with $\eps, \eps^4,|U_i|, H'_i, U_{i+1}$ playing the r\^oles of $\eps, \mu, n, H, U$.
	We obtain a $C_{\ell}$-decomposable subgraph $F_i \subseteq H'_i$ such that $H'_i \setminus H'_i[U_{i+1}] \subseteq F_i$
	and that $\Delta_2(F_i[U_{i+1}]) \leq \eps^4 |U_{i}|$.
	Let $H_{i+1} = H_i[U_{i+1}] \setminus F_i$, we prove it satisfies the required properties.
	
	Clearly $F_i$ is $C_{\ell}$-divisible and $F_i \subseteq H'_i \subseteq H_i$, so \ref{item:decomposition-coverdown-a} implies that $H' - H_{i+1} = (H' - H_{i}) \cup F_i$ has a $C_{\ell}$-decomposition, thus \hyperref[item:decomposition-coverdown-a]{(a$_{i+1}$)} holds.
	From~\ref{item:decomposition-coverdown-d} and $\Delta_2(F_i[U_{i+1}]) \leq \eps^4 |U_{i}| \le \eps^2 |U_{i+1}|$, we have
	$\delta_2(H_{i+1}) \ge (2/3 + 5 \eps) |U_{i+1}| - \eps^2 |U_{i+1}| \ge (2/3 + 4 \eps) |U_{i+1}|$, proving~\hyperref[item:decomposition-coverdown-b]{(b$_{i+1}$)}.
	
	By the properties of $(2/3 + 6 \eps, \eps, m)$-vortices, we have $\deg_{H'}(x, U_{i+2}) \ge (2/3 + 6 \eps) \binom{|U_i|}{2}$ for each $x \in U_{i+1}$,
	together with $\Delta_2(F_i[U_{i+1}]) \leq \eps^2 |U_{i+1}|$ and~\ref{item:decomposition-coverdown-e} we deduce \hyperref[item:decomposition-coverdown-c]{(c$_{i+1}$)} holds, and \hyperref[item:decomposition-coverdown-d]{(d$_{i+1}$)} can be verified similarly.
	Finally, since $F_i \subseteq H'_i = H_i \setminus H_i[U_{i+1}]$, we have $F_i[U_{i+2}]$ is empty therefore $H_{i+1}[U_{i+2}] = H_i[U_{i+2}] = H'[U_{i+2}]$, which verifies~\hyperref[item:decomposition-coverdown-e]{(e$_{i+1}$)}.
	
	Now $H_{\ell} \subseteq H'[U_{\ell}]$ is such that $H'\setminus H_{\ell}$ has a $C_{\ell}$-decomposition.
	
	\medskip
	\noindent \emph{Step 3: Finish.}
	Since both $H'$ and $H' \setminus H_{\ell}$ are $C_{\ell}$-divisible, we deduce $H_{\ell} \subseteq H'[U_\ell]$ is $C_{\ell}$-divisible.
	Therefore,~$H_\ell\in\mathscr L$ and by construction of~$A$ we know that $H_{\ell} \cup A$ is $C_{\ell}$-decomposable.
	Since $H$ is the edge-disjoint union of $H' \setminus H_{\ell}$ and $H_{\ell} \cup A$, and both of them have $C_{\ell}$-decompositions, we deduce $H$ has a $C_{\ell}$-decomposition, as desired.
\end{proof}

\section{Useful tools} \label{section:tools}

We collect various results to be used during the proof of Lemmatas \ref{lemma:absorbinglemma}--\ref{lemma:coverdownlemma}.

\subsection{Counting path extensions} \label{subsection:paths}

The following lemma find short trails between prescribed pairs of vertices.
For a $3$-graph $H$, a set of vertices $U \subseteq V(H)$, and a set of pairs $G \subseteq \binom{V(H)}{2}$ let $\delta_2^{(3)}(H; U, G)$ be the minimum of $|N(e_1) \cap N(e_2) \cap N(e_3) \cap U|$ over all possible choices of $e_1, e_2, e_3 \in G$. 
This is the size of the minimum joint neighbourhood in~$U$ of three distinct pairs in $G$.
Also, let $\delta_2^{(3)}(H; U) = \delta_2^{(3)}(H, U, \binom{V(H)}{2})$ and $\delta_2^{(3)}(H) = \delta_2^{(3)}(H; V(H))$.

\begin{lemma} \label{lemma:therearemanycycles}
    Let $\eps > 0$ and $n, \ell \in \mathbb{N}$ be such that $\ell \ge 5$ and $1/n \ll \eps, 1/\ell$.
    Let $H$ be a $3$-graph on $n$ vertices, $U \subseteq V(H)$ and $G \subseteq \binom{V(H)}{2}$ such that $\{uv\in \binom{V(H)}{2}\colon u\in U\}\subseteq G$. 
    Suppose $\delta_2^{(3)}(H; U, G) \ge 2 \eps n$.
    Then, for every two disjoint pairs $v_1 v_2$ and $v_{\ell -1} v_{\ell}$ in $G$ there exist at least $(\eps n)^{\ell-4}$ many $(v_1, v_2, v_{\ell - 1}, v_{\ell})$-paths on $\ell$ vertices, whose internal vertices are in $U$.
\end{lemma}

\begin{proof}
    Every pair of vertices in $G$ has at least $2 \eps n$ neighbours in $U$.
    For each $1 \leq i \leq \ell - 3$, since $\{uv\in \binom{V(H)}{2}\colon u\in U\}\subseteq G$ we can build a path $v_1 v_2 \dotsm v_{i}$ such that $\{ v_{i-1}, v_{i} \} \in G$ by choosing vertices in $U$ greedily.
    The path is then finished by choosing $v_{\ell-2}$ as a common neighbour in $U$ of the pairs $v_{\ell - 4} v_{\ell - 3}$, $v_{\ell - 3} v_{\ell - 1}$ and $v_{\ell - 1} v_{\ell}$, all of which belong to $G$.
    At any step we only need to avoid choosing one of vertices already chosen so far, which are at most $\ell \leq \eps n$.
    Thus in each step there are at least $\eps n$ possible choices, which gives the desired bound.
\end{proof}

In the particular for a $3$-graph~$H$ with~$\delta_2(H)\geq (2/3+\eps)n$ a simple application of Lemma~\ref{lemma:therearemanycycles} with $U = V(H)$ and $G = \binom{V(H)}{2}$ implies the existence of many trails of length~$\ell\geq 5$ between arbitrary pairs of vertices.

Sometimes we want find many paths which also avoid a small prescribed set of vertices or edges, for instance to extend paths into cycles.
This is accomplished as follows.

\begin{lemma} \label{lemma:wedestroypoquito}
    Let $\eps, \mu > 0$ and $n, \ell \in \mathbb{N}$ be such that $\ell \ge 5$ and $1/n \ll \mu \ll \eps,1/\ell$.
    Suppose that $v_1, v_2, v_{\ell-1}, v_{\ell} \in V(H)$ and there are at least $2 \eps n^{\ell-4}$ many $(v_1, v_2, v_{\ell - 1}, v_{\ell})$-paths on $\ell$ vertices in $H$.
    Let $F \subseteq H$ with $\Delta_2(F) \leq \mu n$.
    Then there are at least $\eps n^{\ell - 4}$ many $(v_1, v_2, v_{\ell - 1}, v_{\ell})$-paths on $\ell$ vertices in $H \setminus F$.
\end{lemma}

\begin{proof}
    The number of $(v_1, v_2, v_{\ell - 1}, v_{\ell})$-paths on $\ell$ vertices such that $v_{1} v_{2} v_3 \in F$ is at most $\deg_F(v_{1} v_{2}) n^{\ell - 5} \leq \Delta_{2}(F) n^{\ell - 5} \leq \mu n^{\ell - 4}$.
	Similar bound are obtained for the paths of the same form such that $v_{\ell - 2} v_{\ell - 1} v_{\ell} \in F$, $v_3 v_4 v_5 \in F$, or $v_{\ell-3} v_{\ell - 2} v_{\ell - 1} \in F$.
	Finally, the paths such that $v_j v_{j+1} v_{j+2} \in F$ for some $3 \leq j \leq \ell - 4$ is at most $\vert E(F)\vert n^{\ell - 7} \leq \mu n^{\ell - 4}$.
	All together, the number of paths destroyed by passing from $H$ to $H \setminus F$ is at most $(\ell - 2) \mu n^{\ell - 4} \leq \eps n^{\ell - 4}$, where the last inequality uses $\mu \ll \eps$.
\end{proof}

The following is an immediate corollary of Lemma~\ref{lemma:therearemanycycles} and Lemma~\ref{lemma:wedestroypoquito}.

\begin{cor} \label{corollary:pathextension}
    Let $\eps > 0$ and $n, \ell, \ell' \in \mathbb{N}$ be such that $1/n \ll \mu \ll \eps \ll \eps',1/\ell, 1/\ell'$ and $\ell \ge \ell'+1$.
    Let $H$ be a $3$-graph on $n$ vertices, $U \subseteq V(H)$ and $G \subseteq \binom{V(H)}{2}$ such that $\{uv\in \binom{V(H)}{2}\colon u\in U\}\subseteq G$.
    Suppose $\delta_2^{(3)}(H; U,G) \ge 2 \eps' n$.
    Let $P$ be a path on $\ell'$ vertices in $H$, whose two endpoints are in $G$.
    Then there are at least $\eps n^{\ell - \ell'}$ many cycles $C$ on $\ell$ vertices which contain $P$, and $V(C) \setminus V(P) \subseteq U$.
\end{cor} 

Observe that for a~$3$-graph~$H$ with~$\delta_2(H)\geq (2/3+\eps)n$ and a set of vertices~$W\subseteq V(H)$ with~$\vert W\vert < \eps n/2$, a simple application of Corollary~\ref{corollary:pathextension}  with~$U=V(H)\setminus W$ and~$G=\binom{V(H)}{2}$ yields the existence of many cycles containing one fix path~$P$ and avoiding the set of vertices~$W$.

\subsection{Probabilistic tools}
We shall use the following concentration inequalities~\cite[Corollary 2.3, Corollary 2.4, Remark 2.5, Theorem 2.10]{RandomGraphs}.

\begin{theorem} \label{theorem:chernoff}
	Let $X$ be a random variable which is a sum of $n$ independent $\{0,1\}$-random variables, or hypergeometric with parameters $n, N, M$.
	\begin{enumerate}[{\upshape (i)}]
	    \item \label{item:chernoff-noexpec} If $x \ge 7 \expectation[X]$, then $\probability[X \ge x] \leq \exp( - x )$,
	    \item \label{item:chernoff-largep} $\probability[ |X - \expectation[X]| \ge t ] \leq 2 \exp( - 2 t^2 / n)$, and
	    \item \label{item:chernoff-smallp} $\probability[ |X - \expectation[X]| \ge t ] \leq 2 \exp( - t^2 / (3 \expectation[X] ))$.
	\end{enumerate}
\end{theorem}

The following lemma allows us to bound the tail probabilities of sums of sequentially-dependent $\{0,1\}$-random variables by comparing them with binomial random variables.
We use the probability-theoretic notion of conditioning in a sequence of random variables, which in our application will take the following form.
If $X_1, \dotsc, X_i$ are random variables, $\probability[X_i = 1 | X_1, \dotsc, X_{i-1}] \leq p_i$ means that the probability of $X_i = 1$ is always at most $p_i$, even after conditioning on any possible output of $X_1, \dotsc, X_{i-1}$.

\begin{theorem} \label{theorem:jain}
	Let $X_1, \dotsc, X_t$ be Bernoulli random variables (not necessarily independent) such that for each $1 \leq i \leq t$ we have $\probability[X_i = 1 | X_1, \dotsc, X_{i-1}] \leq p_i$.
	Let $Y_1, \dotsc, Y_t$ be independent Bernoulli random variables such that $\probability[Y_i = 1] = p_i$ for all $1 \leq i \leq t$.
	If $X = \sum_{i=1}^t X_i$ and $Y = \sum_{i=1}^t Y_i$,
	then $\probability[X \ge k] \leq \probability[Y \ge k]$ for all $k \in \{0, 1, \dotsc, t \}$.
\end{theorem}

The proof of Theorem~\ref{theorem:jain} was given by Jain~\cite[Lemma 7]{Raman1990} in the particular case where $p_i = p$ for all $1 \leq i \leq t$.
The slightly more general statement of Theorem~\ref{theorem:jain} follows by mimicking that proof (which goes by induction on $t$), so we omit it.

\section{Vortex Lemma} \label{section:vortex}

We prove Lemma~\ref{lemma:vortexlemma} by selecting random subsets (cf.~\cite[Lemma 3.7]{BGKLMO2020}).

\begin{proof}[Proof of Lemma~\ref{lemma:vortexlemma}]
    Let $n_0 = n$ and $n_i = \lfloor \xi n_{i-1} \rfloor$ for all $i \ge 1$.
    In particular, note $n_i \leq \xi^i n$.
    Let $\ell$ be the largest $i$ such that $n_i \ge m'$ and let $m = n_{\ell+1}$.
    Note that $\lfloor \xi m' \rfloor \leq m \leq m'$.
    
    Let $\xi_0 = 0$ and, for all $i \ge 1$, define $\xi_i = \xi_{i-1} + 2 (\xi^i n)^{-1/3}$.
    Thus we have
    \[ \xi_{\ell + 1} = 2 n^{-1/3} \sum_{i=1}^{\ell} (\xi^{-1/3})^i
    \leq 2 n^{-1/3} \sum_{i=1}^{\infty} (\xi^{-1/3})^i
    \leq \frac{2 (n \xi)^{-1/3}}{1 - \xi^{-1/3}} \leq \xi, \]
    where in the last inequality we used $1/m' \ll \xi$ and $n \ge m'$.
    
    Note that taking $U_0 = V(H)$ yields a $(\delta - \xi_0, \xi, n_0)$-vortex in $H$.
    Suppose we have already found a $(\delta - \xi_{i-1}, \xi, n_{i-1})$-vortex $U_0 \supseteq \dotsb \supseteq U_{i-1}$ in $H$ for some $i \leq \ell+1$.
    In particular, $\delta_2(H[U_{i-1}]) \ge (\delta - \xi_{i-1})|U_{i-1}|$.
    Let $U_{i} \subseteq U_{i-1}$ be a random subset of size $n_{i}$.
    By Theorem~\ref{theorem:chernoff},
    with positive probability we have, for all $x, y \in U_{i-1}$, $\deg(xy, U_i) \ge (\delta - \xi_{i-1} - n_i^{-1/3})|U_i|$ and $\deg(x, U_i) \ge (\delta - \xi_{i-1} - n_i^{-1/3}) \binom{|U_i|}{2}$.
    Since $\xi_{i-1} + n_i^{-1/3} \leq \xi_i$, we have found a $(\delta - \xi_i, \xi, n_i)$-vortex for $H$.
    In the end, we will have found a $(\delta -\xi_{\ell+1}, \xi, n_{\ell+1})$-vortex for $H$.
    Since we have $m = n_{\ell+1}$ and we have established $\xi_{\ell+1} \leq \xi$, we are done.
\end{proof}

\section{Cover-Down Lemma} \label{section:cover}

\subsection{Extending paths into cycles}

More than once during our proof, we will be faced with the following situation: we have a family of (not too many) edge-disjoint tight paths, and we want to extend each of these paths into a tight cycle of a given length, such that all of the obtained cycles are edge-disjoint.
In this subsection we will prove a lemma which will find such extensions for us.

Given a path~$P$ we say that a path or a cycle~$C$ is an \emph{extension of~$P$} if~$P\subseteq C$.
Let~$H$ be a~$3$-graph, for a path~$P\subseteq H$ and a pair of vertices $e \in \binom{V(H)}{2}$ we say that $P$ is of \emph{type~$r$} for~$e$, where $r=\max\{e\cap s(P), e\cap t(P)\}$.
The only possible types are~$0$, $1$, or $2$.

We say that a collection of edge-disjoint paths $\cP$ in~$H$ is \emph{$\gamma$-sparse} if, for each $e \in \binom{V(H)}{2}$ and each $r \in \{0,1, 2\}$, $\cP$ has at most $\gamma n^{3-r}$ paths $P$ of type $r$ for $e$.

\begin{lemma}[Extending Lemma] \label{lemma:extending}
	Let $\eps, \mu, \gamma > 0$ and $n, \ell, \ell' \in \mathbb{N}$ such that $\ell' \ge 4$, $\ell \ge \ell' + 2$ and $1/n \ll \gamma \ll \mu \ll \eps, 1/\ell$.
	Let $H_1, H_2$ be two edge-disjoint $3$-graphs on the same vertex set $V$ of size $n$.
	Let $P$ be the $3$-uniform tight path on $\ell'$ vertices,
	and let $\cP = \{ P_1, \dotsc, P_t \} $ be an edge-disjoint collection of copies of $P$ in $H_1$ such that
	\begin{enumerate}[{\upshape (F1)}]
		\item\label{item:extending-Fsparse} $\cP$ is $\gamma$-sparse, and
		\item \label{item:extending-manyextensions} for each $P_i \in \cP$, there exists at least $2 \eps n^{\ell - \ell'}$ copies of $C_{\ell}$ in $H_1 \cup H_2$ which extend $P_i$ using extra edges of $H_2$ only.
	\end{enumerate}
	Then, there exists a~$C_\ell$-decomposable subgraph~$F\subseteq H_1\cup H_2$, such that
	\begin{enumerate}[{\upshape (C1)}]
		\item \label{item:extending-weextend} $E(\cP) \subseteq F$, and 
		\item \label{item:extending-wesparse} $\Delta_2(F \setminus E(\cP)) \leq \mu n$.
	\end{enumerate}
\end{lemma}

\begin{proof}
	The idea is to pick, sequentially, an extension $C_i$ of $P_i$ into an $\ell$-cycle, chosen uniformly at random among all the extensions which do not use edges already used by $C_1, \dotsc, C_{i-1}$.
	Since $\cP$ is $\gamma$-sparse and there are plenty of choices for $C_i$ in each step, we expect that in each step the random choices do not affect the codegree of the graph formed by the unused edges in $H_2$ by much.
	This will ensure that, even after removing the edges used by $C_1, \dotsc, C_{i-1}$, there are still many extensions available for $P_i$.
	If all goes well, then we can continue the process until the end, thus achieving~\ref{item:extending-weextend} and~\ref{item:extending-wesparse} by setting $F = \bigcup_{1 \leq i \leq t} E(C_i)$.
	
	To formalise the above plan, we begin by noting that the removal of a sufficiently sparse $3$-graph from $H_2$, there are still many extensions available for each $P_i$.
	Given $G \subseteq H_2$ and $1 \leq i \leq t$,
	let $\ccC_i(G)$ be the set of \emph{$G$-avoiding cycle-extensions of $P_i$}, that is, the copies of $C_{\ell}$ in $H_1 \cup H_2$ which extend $P_i$ and use extra edges from $H_2 \setminus G$ only.
	By assumption, $|\ccC_i(\emptyset)| \ge 2 \eps n^{\ell - \ell'}$, thus Lemma~\ref{lemma:wedestroypoquito} implies that
	\begin{align}
		\text{if $G \subseteq H_2$ is such that $\Delta_2(G) \leq \mu n$, then $|\ccC_i(G)| \ge \eps n^{\ell - \ell'}$.} \label{claim:extending-sparseremoval}
	\end{align}
	
	We now describe the random process which outputs edge-disjoint extensions $C_i$ of $P_i$ for each $1 \leq i \leq t$.
	In the case of success each $C_i$ will be an $\ell$-cycle extending $P_i$.
	To account for the case of failure, in our description we will allow the degenerate case in which $C_i \setminus P_i$ is empty.
	
	For each $1 \leq i \leq t$, assume we have already chosen $C_1, C_2, \dotsc, C_{i-1} \subseteq H_1 \cup H_2$ edge-disjoint graphs, and we describe the choice of $C_i$.
	Let $G_{i-1} = \bigcup_{1 \leq j < i} E(C_j) \setminus E(P_j)$ correspond to the edges of $H_2$ used by the previous choices of $C_j$, which we need to avoid when choosing $C_i$
	(note that $G_0$ is empty).
	If $\Delta_2(G_{i-1}) \leq \mu n$, then by~\eqref{claim:extending-sparseremoval} we have $|\ccC_i(G_{i-1})| \ge \eps n^{\ell - \ell'}$, and we take~$C_i\in \ccC_i(G_{i-1})$ uniformly at random.
	Otherwise, if $\Delta_2(G_{i-1}) > \mu n$, let $C_i = P_i$.
	
	In any case, the process outputs a collection $C_1, \dotsc, C_{t}$ of edge-disjoint cycle or paths which extend $P_i$.
	Our task now is to show that with positive probability, there is a choice of $C_1, \dotsc, C_{t}$ such that $\Delta_2(G_{t}) \leq \mu n$.
	This would imply also that each $C_i$ was an $\ell$-cycle.
	Formally, for each $1 \leq i \leq t$, let $\cS_i$ be the event that $\Delta_2(G_{i}) \leq \mu n$.
	Thus it is enough to show $\probability[\cS_t] > 0$.
	
	Fix $e \in \binom{V}{2}$.
	For each $1 \leq i \leq t$,
	let $X_i(e)$ be the random variable which takes the value $1$ precisely if $e$ belongs to an edge of $C_i \setminus P_i$, and $0$ otherwise.
	Equivalently, $X_i(e) = 1$ if and only if $e$ belong to the shadow $\partial (C_i \setminus P_i)$.
	Since $\Delta_2(C_i) \leq 2$ for each $1 \leq i \leq t$, we have
	\begin{align}
		\deg_{G_i}(e) \leq 2 \sum_{j=1}^i X_j(e). \label{equation:degreeui}
	\end{align}
	For each $1 \leq i \leq t$, define \[p^\ast_i(e) := \min\left\{ 1, \frac{c}{n^{2-r}}\right\},\] where $r \in \{0,1,2\}$ is such that $P_i$ is of type $r$ for $e$, and $c := 4 \ell \eps^{-1}$.
	
	\begin{claim}
		For each $e \in \binom{V}{2}$ and $1 \leq i \leq t$,
		\[ \probability[X_i(e) = 1 | X_1(e), X_2(e), \dotsc, X_{i-1}(e) ] \leq p^\ast_i(e), \]
	\end{claim}
	
	\begin{proofclaim}[Proof of the claim.]
		Using conditional probabilities, we separate our analysis depending on whether $\cS_{i-1}$ holds or not.
		Assume first that $\cS_{i-1}$ fails.
		Then the process declares $C_i = P_i$, thus $C_i \setminus P_i$ is empty.
		Therefore $X_i(e) = 0$ regardless of the values of $X_1(e), \dotsc, X_{i-1}(e)$, and we have
		\[ \probability[X_i(e) = 1 | X_1(e), X_2(e), \dotsc, X_{i-1}(e), \cS_{i-1}^c ] = 0 \leq p^\ast_i(e). \]
		
		Now assume that $\cS_{i-1}$ holds.
		Then the set $G_{i-1}$ of edges to be avoided while constructing $C_i$ satisfies $\Delta_2(G_{i-1}) \leq \mu n$.
		By~\eqref{claim:extending-sparseremoval}, $C_i$ will be an $\ell$-cycle extending $P_i$ selected uniformly at random from the set $\ccC_i(G_{i-1})$, which has size at least $\eps n^{\ell - \ell'}$;
		and this will happen no matter the values of $X_1(e), \dotsc, X_{i-1}(e)$.
		
		If $P_i$ is of type $2$ for $e$,
		then we are required to bound a probability by $p^\ast_i(e) = 1$, which holds trivially.
		Suppose now that $P_i$ is of type $0$ for $e$, and suppose $P_i = v_1 v_2 \dotsm v_{\ell'}$.
		For $C_i \in \ccC_i(G_{i-1})$, $C_i \setminus P_i$ is a path of the form $v_{\ell'-1} v_{\ell'} u_1 u_2 \dotsb u_{\ell - \ell'} v_1 v_2$.
		We wish to estimate the number of such paths where $e \in \partial(C_i \setminus P_i)$.
		Since $P_i$ is of type $0$ for $e$, then $e \in \partial(C_i \setminus P_i)$ can only happen if $e = u_j u_k$ for $|j-k| \leq 2$.
		There are $(\ell - \ell' - 1) - (\ell - \ell' - 2) \leq 2 \ell$ choices for $j, k$.
		Having fixed those, there are two $2$ possibilities for assigning $e$ to $\{u_j, u_k\}$,
		and having fixed those, there are at most $n$ possibilities for each other $u_p$ with $p \notin \{j, k\}$.
		All together, the number of $C_i$ which extend $P_i$ and such that $e \in \partial(C_i \setminus P_i)$ is certainly at most $4 \ell n^{\ell - \ell' - 2}$.
		Thus we have \[ \probability[X_i(e) = 1 | X_1(e), X_2(e), \dotsc, X_{i-1}(e), \cS_{i-1} ] \leq \frac{4 \ell n^{\ell'-\ell-2}}{|\ccC_i(G_{i-1})|} \leq \frac{4 \ell}{\eps n^2} = \frac{c}{n^2} = p^\ast_i(e), \]
		as required.
		Finally, if $P_i$ is of type $1$ for $e$, then similar (but simpler) calculations show that $\probability[X_i(e) = 1 | X_1(e), X_2(e), \dotsc, X_{i-1}(e), \cS_{i-1} ] \leq \frac{6 n^{\ell'-\ell-1}}{|\ccC_i(G_{i-1})|} \leq \frac{c}{n} = p^\ast_i(e)$,
		and we are done.		
	\end{proofclaim}

	Now, we use that $\cP$ is $\gamma$-sparse to argue $\sum_{i=1}^t p_i^\ast(e)$ is suitably small.
	Indeed, for each $r \in \{0,1,2\}$, let $t_r$ be the number of $i \in \{1, \dotsc, t\}$ such that $P_i$ is of type $r$ for $e$.
	Since $\cP$ is $\gamma$-sparse, we have $t_r \leq \gamma n^{3 - r}$ for each $r \in \{0,1,2\}$.
	Therefore, we have
	\begin{align}
		\sum_{i=1}^t p^\ast_i(e)
		= t_0 \frac{c}{n^2} + t_1 \frac{c}{n} + t_2
		\leq \gamma c n + \gamma c n + \gamma n \leq \frac{\mu}{30} n, \label{equation:pibound}
	\end{align}
	where the last inequality follows from the choice of $c$ and $\gamma \ll \mu , \eps$.
	
	We now claim that 
	\begin{align}
		\probability\left[ \sum_{i=1}^t X_i(e) \geq \frac{\mu}{3} n \right] \leq \exp\left( - \frac{\mu}{3} n \right).
		\label{equation:extending-concentration}
	\end{align}
	Indeed, inequality \eqref{equation:pibound} implies that $7 \sum_{i=1}^t p^\ast_i(e) \leq \mu n / 3$,
	so the bound follows from Theorem~\ref{theorem:jain} combined with Theorem~\ref{theorem:chernoff}.
	
	For each $e \in \binom{V(H)}{2}$, let $X_e := \sum_{i=1}^t X_i(e)$.
	Let $\mathcal{E}$ be the event that $\max_e X_e \leq \mu n / 3$.
	By using an union bound over all the (at most $n^2$) possible choices of $e$ and using \eqref{equation:extending-concentration}, we deduce that $\mathcal{E}$ holds with probability at least $1 - o(1)$.
	
	Now we can show that $\cS_t$ holds with positive probability.
	We shall prove that $\probability[\cS_t | \mathcal{E}] = 1$, which then will imply $\probability[\cS_t] \ge \probability[\cS_t | \mathcal{E}] \probability[\mathcal{E}] \ge 1 - o(1)$.
	So assume $\mathcal{E}$ holds, that is, $\max_e X_e \leq \mu n / 3$.
	Note that $\cS_0$ holds deterministically, and suppose $1 \leq i \leq t$ is the minimum such that $\cS_i$ fails to hold.
	Since $\cS_{i-1}$ holds, using~\eqref{equation:degreeui} we deduce
	\begin{align*}
		\Delta_2( G_{i} )
		& \leq 2 + \Delta_2(G_{i-1}) = 2 + \max_{e} \deg_{G_{i-1}}(e)
		\leq 2 \left ( 1 + \max_e \sum_{j=1}^{i-1} X_i(e) \right) \\
		& \leq 2 \left ( 1 + \max_e X_e \right) \leq 2 \left(1 + \frac{\mu}{3}n \right) \leq \mu n,
	\end{align*}
	where in the second to last inequality we used $\mathcal{E}$,
	and in the last inequality we used $1/n \ll \mu$.
	Thus $\cS_i$ holds, a contradiction.
\end{proof}

\subsection{Well-behaved approximate cycle decompositions} \label{subsection:wellbehaved}

In this section we show the existence of approximate cycle decomposition which are `well-behaved', meaning that the subgraph left by the uncovered edges has small codegree.
The argument is different depending on the two setting considered by Theorem~\ref{theorem:ldecomposcodegree}, and we start with the former.

When $\ell$ is divisible by $3$, the tight cycle $C_{\ell}$ is $3$-partite.
By a well-known theorem from Erd\H{o}s~\cite[Theorem 1]{Erdos1964}, we know that the Tur\'an number of $C_{\ell}$ is degenerate, i.e. edge-maximal $C_{\ell}$-free $3$-graphs on $n$ vertices have at most $o(n^3)$ edges.
This allows us to find an approximate decomposition of any $3$-graph $H$ with copies of $C_{\ell}$ if $\ell$ is divisible by~$3$, simply by removing copies of $C_{\ell}$ greedily until $o(n^3)$ edges remain. 
This argument alone does not provide us with the `well-behavedness' condition we alluded to earlier, but it is, however, possible to modify such a packing locally to guarantee such a property holds.

\begin{lemma}[Well-behaved approximate cycle decompositions, version 1] \label{lemma:wellbehaved}
	Let $\eps, \gamma > 0$ and $n, \ell \in \mathbb{N}$ be such that $\ell \ge 9$ is divisible by $3$ and $1/n \ll \eps, \gamma, 1/\ell$.
	Let $H$ be a $3$-graph on $n$ vertices with $\delta_2(H) \ge (2/3 + \eps)n$.
	Then $H$ has a $C_{\ell}$-packing $\mathcal{C}$ such that $\Delta_2(H \setminus E(\mathcal{C})) \leq \gamma n$.
\end{lemma}

Results in a similar spirit were proven in~\cite{BKLO2016}.
The proof is not difficult but somewhat long and repetitive, thus we defer it to Appendix~\ref{appendix:wellbehaved}.

Now we consider the second range of $\ell$, where it $\ell \ge 10^7$, in which we can show the following.

\begin{lemma}[Well-behaved approximate cycle decomposition, version 2] \label{lemma:wellbehaved2}
	Let $\eps, \gamma > 0$ and $n, \ell \in \mathbb{N}$ be such that $\ell \ge 10^7$ and $1/n \ll \eps, \gamma, 1/\ell$.
	Let $H$ be a $3$-graph on $n$ vertices with $\delta_2(H) \ge (2/3 + \eps)n$.
	Then $H$ has a $C_{\ell}$-packing $\mathcal{C}$ such that $\Delta_2(H \setminus E(\mathcal{C})) \leq \gamma n$.
\end{lemma}

In this range we exploit the connection of \emph{fractional graph decompositions} with their integral counterparts.
Given a $3$-graph $H$, let $\mathcal{C}_{\ell}(H)$ be the family of all $\ell$-cycles in $H$, and given $X \in E(H)$ let $\mathcal{C}_{\ell}(H, X) \subseteq \mathcal{C}_{\ell}(H)$ be those cycles which use  the edge $X$.
A \emph{fractional $C_{\ell}$-decomposition} of a $3$-graph $H$ is a function $\omega: \mathcal{C}_{\ell}(H) \rightarrow [0,1]$ such that for every edge $X \in H$ we have $\sum_{C \in \mathcal{C}_{\ell}(H, X)} \omega(C) = 1$.
Joos and Kühn~\cite{KuhnJoos2021} proved the existence of fractional $C^k_{\ell}$-decompositions under general conditions.
We state their results only in the particular case $k = 3$.
A $3$-graph $H$ on $n$ vertices is \emph{$(\alpha, \ell)$-connected} if for every two ordered edges $(s_1, s_2, s_3)$, $(t_1, t_2, t_3) \in V(H)^3$, there are at least $\alpha n^{\ell - 1} / ( 3! |E(H)|)$ walks with $\ell$ edges starting at $(s_1, s_2, s_3)$, ending at $(t_1, t_2, t_3)$.

\begin{theorem}[Joos and Kühn~\cite{KuhnJoos2021}] \label{theorem:fractional}
    For all $\alpha \in (0,1)$, $\mu \in (0, 1/3)$ and $\ell \ge 2$,
    there is $n_0$ such that the following holds for all $n \ge n_0$.
    Suppose $H$ is an $(\alpha, \ell_0)$-connected $3$-graph on $n$ vertices with $540 \frac{\ell_0}{\alpha} \log \frac{\ell_0}{\alpha} \log \frac{1}{\mu} \le \ell$.
    Then there is a fractional $C_{\ell}$-decomposition $\omega$ of $H$ with
    \[ (1 - \mu) \frac{2 |E(H)|}{\Delta(H)^{\ell}} \leq \omega(C) \leq (1 + \mu) \frac{2 |E(H)|}{\delta(H)^{\ell}} \]
    for all $\ell$-cycles $C$ in $H$.
\end{theorem}

To use this theorem, we show that $3$-graphs with $\delta_2(H) \ge 2n/3$ are $(\alpha, \ell_0)$-connected for some suitable $\alpha, \ell_0$.
The following argument is due to Reiher~\cite[Lemma 2.3]{KuhnJoos2021}.
We include it for completeness and since for $k = 3$ one can give a better value of $\alpha$, which in turn increases the range of $\ell$ in which one can apply Theorem~\ref{theorem:fractional}.

\begin{lemma} \label{lemma:reiher}
For each~$d \ge 1/2$, every $3$-graph $H$ on $n$ vertices and such that $\delta_2(H) \ge (d + o(1))n$ is $(d^2(2d-1)^4, 8)$-connected.
\end{lemma}

\begin{proof}
Let $V = V(H)$ and $(s_1, s_2, s_3), (t_1, t_2, t_3) \in V^3$ be two arbitrary ordered edges of $H$.
For $z \in V(H)$, let the function $I_z : V^2 \rightarrow \{0,1\}$ be such that $I_z(x_1, x_2) = 1$ if and only if $s_2 s_3 x_1 x_2 t_1 t_2$ is a path in the link-graph of $z$ in $H$.
Let $N = N_H(s_2 s_3) \cap N_H(t_1t_2)$ and note that~$\vert N\vert > (2d-1)n$.
Note that if $z_1, z_2 \in N$ (possibly equal) and $(x_1, x_2) \in V^2$ are such that $I_{z_1}(x_1, x_2) = I_{z_2}(x_1, x_2) = 1$, then $s_1 s_2 s_3 z_1 x_1 x_2 z_2 t_1 t_2 t_3$ is a walk from $(s_1, s_2, s_3)$ to $(t_1, t_2, t_3)$ using $8$ edges, call such walks \emph{standard}.

First, note that having fixed $z\in N$, the number of $(x_1, x_2) \in V^2$ such that $I_z(x_1, x_2) = 1$ can be bounded as follows:
choose $x_1 \in N_H(s_3 z)$ arbitrarily (there are at least $dn$ choices) and then $x_2 \in N_H(z x_1) \cap N_H(z t_1)$ (of which there are at least $(2d-1)n$ choices).
Thus we have $\sum_{(x_1, x_2) \in V^2} I_z(x_1, x_2) \ge d(2d-1)n^2$ for all $z \in N$.

On the other hand, note that for a fixed $(x_1, x_2)$ with $x_1 \neq x_2$, the number of standard walks which use $(x_1, x_2)$ is exactly $(\sum_{z \in N} I_z(x_1, x_2))^2$.
Thus the number of standard walks is at least (using Jensen's inequality in the first inequality, and $|N| \ge (2d-1)n$ in the third inequality)
\begin{align*}
    \sum_{(x_1, x_2) \in V^2} 
    \left( \sum_{z \in N} I_z(x_1, x_2) \right)^2
    & \ge
    n^2 \left( \frac{1}{n^2} \sum_{z \in N} 
    \sum_{(x_1, x_2) \in V^2} I_z(x) \right)^2 \\
    & \ge
    n^2 \left( \frac{1}{n^2} \sum_{z \in N} 
    d(2d-1) n^2 \right)^2
    \ge 
    d^2(2d-1)^4 n^4,
\end{align*}
as required.
\end{proof}

To prove Lemma~\ref{lemma:wellbehaved2}, we combine the fractional matching of Theorem~\ref{theorem:fractional} with a nibble-type matching argument.
We use a result of Alon and Yuster~\cite{AlonYuster2005} (but see also Kahn~\cite{Kahn1996} and Ehard, Glock and Joos~\cite{EhardGlockJoos2020} for variations and extensions).

\begin{proof}[Proof of Lemma~\ref{lemma:wellbehaved2}]
    Let $\alpha = 4 \times 3^{-6}$ (as in Lemma~\ref{lemma:reiher} for~$d=2/3$) and $\ell_0 = 8$.
    By Lemma~\ref{lemma:reiher}, $H$ is $(\alpha, \ell_0)$-connected.
    A numerical calculation shows that we can fix $\mu \in (0, 1/3)$ such that $540 \frac{\ell_0}{ \alpha} \log \frac{\ell_0}{\alpha} \log \frac{1}{\mu} \leq 10^7 \leq \ell$.
    Thus Theorem~\ref{theorem:fractional} informs us that there exists a fractional $C_{\ell}$-decomposition $\omega$ of $H$ with
    \[ \omega(C) \leq (1 + \mu) \frac{2 |E(H)|}{\delta_{2}(H)^{\ell}} \leq 4 \frac{ |E(H)|}{\delta_{2}(H)^{\ell}} \leq \frac{4 n^3}{\delta_2(H)^\ell}
    \leq \frac{4 \times 3^{\ell}}{n^{\ell - 3}} \]
    for all $C \in \mathcal{C}_{\ell}(H)$.
    
    Consider the auxiliary $\ell$-uniform hypergraph $F$ with vertex set $E(H)$, and an edge for each cycle in $\mathcal{C}_{\ell}(H)$ corresponding to its set of $\ell$ edges.
    Define a random subgraph $F' \subseteq F$ by keeping each edge $C$ with probability $p_C := n^{1/2} \omega(C)$.
    By the bounds on $\omega(C)$ and $1/n \ll 1/\ell$ we have $p_C \leq 1$ for all $C \in \mathcal{C}_{\ell}(H)$.
    For each edge $e \in E(H)$, we have $\expectation[ \deg_{F'}(e)] = n^{1/2} \sum_{C \in \mathcal{C}_\ell(H, e)} \omega(C) = n^{1/2}$.
    Two distinct edges $e, f \in E(H)$ can participate together in at most $O(n^{\ell - 4})$ $\ell$-cycles in $H$, thus we have $\expectation[ \deg_{F'}(e, f)] = O(n^{-1/2})$.
    Standard concentration inequalities (Theorem~\ref{theorem:chernoff}\ref{item:chernoff-noexpec} and~\ref{item:chernoff-smallp}), imply that with very high probability $F'$ satisfies $\deg_{F'}(e) = (1 + o(1))n^{1/2}$ for each $e \in V(F')$, and thus $\delta_1(F') \ge (1 - o(1))\Delta_1(F')$; and moreover $\Delta_2(F') = o(n^{1/2})$.
    
    For each $2$-set $uv$ of vertices of $H$, let $H_{uv} \subseteq V(F)$ correspond to the edges in $H$ which contain $uv$.
    There are at most $n^{2}$ such sets and each has size at least $2n/3$.
    Thus, the Alon--Yuster theorem~\cite{AlonYuster2005} implies the existence of a matching $M$ in $F'$ such that at most $\gamma n$ vertices in $V(F')$ are uncovered in each $H_{uv}$.
    The matching $M$ in $F' \subseteq F$ translates to a $C_{\ell}$-packing $\mathcal{C}$ in $H$, and the latter condition implies $\Delta_2(H \setminus E(\mathcal{C})) \leq \gamma n$, as desired.
\end{proof}

\subsection{Proof of the Cover-Down Lemma}

As a final tool, we borrow the following theorem of Thomassen~\cite{Thomassen2008} about path-decompositions of graphs.

\begin{theorem}[\cite{Thomassen2008}] \label{theorem:thomassen}
	Any $171$-edge-connected graph $G$ such that $|E(G)|$ is divisible by $3$ has a $P_3$-decomposition.
\end{theorem}

\begin{proof}[Proof of Lemma~\ref{lemma:coverdownlemma}]
	Let $\gamma_1, p_1, p_2 > 0$ such that $\gamma_1 \ll p_1 \ll p_2 \ll \mu, \eps$.
	For $i \in \{0, 1, 2, 3\}$,
	say an edge $e$ of $H$ is of \emph{type $i$} if $|e \cap U| = i$,
	and let $H_i \subseteq H$ be the edges of $H$ which are of type $i$.
	For $i \in \{1,2\}$,
	let $R_i \subseteq H_i$ be defined by choosing edges independently at random from $H_i$ with probability $3 p_i / 2$.
	By assumption, $\delta_2^{(3)}(H; U) \ge 3 \eps |U|$ (see definition at the beginning of Section~\ref{subsection:paths}).
	
	By Theorem~\ref{theorem:chernoff} we get that, for $i \in \{1,2\}$, with non-zero probability, that
	\begin{align}
	    \Delta_2(R_i) & \leq 2 p_i n, \label{equation:coverdown-codeg} \\
	   \delta_2^{(3)}(R_1 \cup R_2 \cup H[U]; U) & \ge 2 \eps p_1 |U|, \text{ and} \label{equation:coverdown-pairU-1} \\
	    \delta_2^{(3)}( R_2 \cup H[U]; U, G) & \ge 2 \eps p_2 |U|, \label{equation:coverdown-pairU-2}
	\end{align}
	where $G \subseteq \binom{V(H)}{2}$ corresponds to the pairs $e$ such that $e \cap U \neq \emptyset$.
	From now on we assume $R_1, R_2$ are fixed with those properties.
	
	Let $H' = H - H[U] - R_1 - R_2$.
	Recall that, by assumption, $\delta_2(H) \ge (2/3 + 2 \eps)n$ and $|U| = \lfloor \eps n \rfloor$.
	By our choice of $p_1, p_2 \ll \eps, \mu$ and \eqref{equation:coverdown-codeg}, we deduce that $\delta_2(H') \ge (2/3 + \eps/2)n$.
	
	We consider two possible cases depending on the value of $\ell$.
	If $\ell \ge 9$ is divisible by $3$, then we apply Lemma~\ref{lemma:wellbehaved}, otherwise by assumption $\ell \ge 10^7$, and we can apply Lemma~\ref{lemma:wellbehaved2}.
	In any case, the output is a $C_{\ell}$-packing $\cC$ in $H'$ such that $\Delta_2(H' \setminus E(\cC)) \leq \gamma_1 n$.
	Let $J = H' \setminus E(\cC)$ be the edges in $H'$ not covered by $\cC$, and for each $i \in \{0, 1, 2 \}$ let $J_i$ be the edges of type $i$ in $J$.
	We shall cover the edges in~$J$ with cycles of length~$\ell$ and for that we will proceed in three steps, covering the edges of $J_0$,~$J_1$, and~$J_2$ in order.
	
	Consider each edge in $J_0$ as a path on three vertices $v_1 v_2 v_3$, assigning to each edge an arbitrary order.
	Let $\cP_0$ be the collection of those paths.
	The inequalities $\Delta_2(J_0) \leq \Delta_2(J) \leq \gamma_1 n$ show that $\cP_0$ is $\gamma_1$-sparse.
	Let $\mu_1, \eps_1 > 0$ satisfy $\gamma_1 \ll \mu_1 \ll \eps_1 \ll p_1, \eps$.
	Equation~\eqref{equation:coverdown-pairU-1} and Corollary~\ref{corollary:pathextension} imply that each $P \in \cP_0$ can be extended to at least $2 \eps_1 n^{\ell - 3}$ cycles $C$, such that $C \setminus P \subseteq R_1 \cup R_2 \cup H[U]$ and $V(C) \setminus V(P) \subseteq U$.
	Then an application of Lemma~\ref{lemma:extending} with $\eps_1, \mu_1, 3, J_0, R_1 \cup R_2 \cup H[U], \cP_0$ in place of 
	$\eps, \mu, \ell', H_1, H_2, \cP$ respectively, implies that there is a~$C_\ell$-decomposable subgraph~$F_0$ such that~$F_0\supseteq J_0$, and
	\begin{align}
	   \Delta_2(F_0 \setminus J_0) \leq \mu_1 n.
	   \label{equation:coverdown-firstsparse}
	\end{align}
	By construction,~$F_0$ is edge-disjoint with the cycles in~$\cC$, and then $F'_0 = E(\cC)\cup F_0$ is~$C_\ell$-descomposable. 
	Note that all edges not covered by~$F_0'$ lie in $(J_1 \cup J_2) \cup (R_1 \cup R_2) \cup H[U]$.
	
	Let $J'_1 = (J_1 \cup R_1) \setminus F'_0$ and $R'_2 = (R_2 \cup H[U]) \setminus F'_0$.
	Let $\gamma_2, \mu_2, \eps_2 > 0$ be such that $p_1 \ll \gamma_2 \ll \mu_2 \ll \eps_2 \ll p_2, \eps$.
	Since $J'_1 \subseteq J_1 \cup R_1 \subseteq J \cup R_1$, we have 
	$$\Delta_2(J'_1) \leq \Delta_2(J) + \Delta_2(R_1) \leq \gamma n + 2 p_1 n \leq \gamma_2 n.$$
	Since each edge in $J'_1$ is of type $1$ in $H$, we can consider each edge in $J'_1$ as a path $P = v_1 v_2 v_3$ where $v_2 \in U$ and $v_1, v_3 \notin U$; and let $\cP_1$ be the collection of those paths.
	Then $\Delta_2(J'_1) \leq \gamma_2 n$ implies $\cP_1$ is $\gamma_2$-sparse.
	By \eqref{equation:coverdown-pairU-2} and  \eqref{equation:coverdown-firstsparse},
	together with Corollary~\ref{corollary:pathextension},
	we deduce that each $P \in \cP_1$ can be extended to at least $2 \eps_2 n^{\ell - 3}$ cycles $C$, such that $C \setminus P \subseteq R'_2$ and $V(C) \setminus V(P) \subseteq U$.
	Apply Lemma~\ref{lemma:extending} with $\eps_2, \mu_2, \gamma_2, 3, J'_1, R'_2, \cP_1$ in place of $\eps, \mu, \gamma, \ell', H_1, H_2, \cP$ to obtain a $C_{\ell}$-decomposable subgraph~$F_1$ such that~$F_1\supseteq J_1'$, and
	\begin{align}
	   \Delta_2(F_1 \setminus J_1) \leq \mu_2 n.
	   \label{equation:coverdown-secondsparse}
	\end{align}
	By construction, $F_1$ and $F_0'$ are edge-disjoint, and then~$F'_1 = F_1\cup F_0'$ is~$C_\ell$-decomposable. 
	Note that the edges not covered by $F'_1$ lie in $J_2 \cup R_2 \cup H[U]$.
	
	Let $J'_2 = (J_2 \cup R_2) \setminus F'_1$.
	Note that each edge in $J'_2$ is of type $2$.
	For each $v \in V(H) \setminus U$, let~$G_v = J'_2(v,U)$, that is, $G_v$ is the link graph of~$v$ in~$J'_2$ restricted to~$U$.
	Fix $v \in V(H) \setminus U$.
	Given $x, y \in U$, the equations~\eqref{equation:coverdown-pairU-2} and \eqref{equation:coverdown-secondsparse} imply that $x$ and $y$ have at least $2 \eps p_2 |U| - 2 \mu_2 n \ge 171$ common neighbours in $G_v$, so $G_v$ is $171$-edge-connected.
	Since $v \notin U$,
	our assumption on $H$ implies that the number of edges of $H(v)$ is divisible by $3$.
	Note that $G_v$ is exactly the link-graph over $H \setminus F'_1$ when restricted to $U$.
	Therefore, and since $F'_1$ is $C_{\ell}$-decomposable, the number of edges in $G_v$ is divisible by $3$ as well.
	
	By Theorem~\ref{theorem:thomassen}, $G_v$ has a decomposition into paths $\cP'_v = \{P_1, \dotsc, P_t\}$, each of length $3$.
	Observe that these paths yields to a collection of ($3$-uniform) paths in~$J_2'$ by substituting each path $P_i = w_1 w_2 w_3 w_4$ in $\cP'_v$ by the tight path $w_1 w_2 v w_3 w_4$.
	Let $\cP_v$ be the collection of paths obtained in this way.
	Observe that for $u \neq v$ in $V(H) \setminus U$, $\cP_v$ and $\cP_u$ are edge-disjoint.
	Let $\cP_2 = \bigcup_{v \in V(H) \setminus U} \cP_v$.
	Note that $\cP_2$ decomposes $J'_2$ into paths on five vertices.
	
	Let $\gamma_3, \eps_3 > 0$ be such that $p_2 \ll \gamma_3 \ll \eps_3 \ll \mu_3 \ll \mu, \eps$.
	Recall that $|U| = \lfloor \eps n \rfloor$.
	Since $J'_2 \subseteq J_2 \cup R_2 \subseteq F \cup R_2$, we have $\Delta_2(J'_2) \leq \Delta_2(R_2) + \Delta_2(J) \leq 2 p_2 n + \gamma_1 n \leq \gamma_3 n$, so $\cP_2$ is $\gamma_3$-sparse.
	Let $H'_2 = H[U] \setminus F'_1$.
	We have $F'_1[U] = F_1[U] \cup F_0[U]$.
	By \eqref{equation:coverdown-firstsparse}--\eqref{equation:coverdown-secondsparse}, we have $\delta_2(H'_2) \ge \delta_2(H[U]) - 2 \mu_2 n \ge (2/3 + \eps/2) |U|$.
	By Corollary~\ref{corollary:pathextension}, we deduce each $P \in \cP_2$ can be extended to at least $2 \eps_2 n^{\ell - 5}$ cycles $C$ such that $C \setminus P \subseteq H'_2$.
	Thus we can apply Lemma~\ref{lemma:extending} with $\eps_3,\mu_3, \gamma_3, 5, J'_2, H'_2, \cP_2$ playing the rôles of $\eps, \mu, \gamma_3, \ell', H_1, H_2, \cP$ respectively, to obtain a~$C_\ell$-decomposable subgraph~$F_2$ such that~$F_2\supseteq J'_2$, and
	\begin{align}\label{equation:coverdown-thirdsparse}
	    \Delta_2(F_2 \cap H'_2) \leq \mu_3 n.
	\end{align}
	By construction, $F_2$ and $F_1'$ are edge-disjoint,
	and then~$F=F_1'\cup F_2$ is~$C_\ell$-decomposable.
	Moreover, all edges not contained in~$U$ are covered by~$F$.
	In fact, we have that
	$$H - H[U] =E(\cC) \cup J_0 \cup (J_1 \cup R_1) \cup (J_2 \cup R_2) \subseteq E(\cC) \cup F_0 \cup F_1 \cup F_2 = F.$$ %
	Finally, inequalities \eqref{equation:coverdown-firstsparse}--\eqref{equation:coverdown-thirdsparse} yield that $\Delta_2(F[U]) \leq \mu n$, as required.
\end{proof}

\section{Absorbing Lemma} \label{section:absorbing}

In this section we prove Lemma~\ref{lemma:absorbinglemma}. 
We need to show that, given a sufficiently large~$H$ with~$\delta_2(H)\geq (2/3+\eps)n$ and a subgraph~$R\subseteq H$ on at most~$m$ vertices, there is an~$C_\ell$-absorber~$A$ for~$R$ on at most~$O(m^9\ell^9)$ vertices. 
We divide the proof in two main parts. 

First, in Section~\ref{subsection:tours} we shall find a bounded-size hypergraph~$A_1\subseteq H$, edge-disjoint from~$R$, which admits a~$C_\ell$-decomposition. 
This subgraph will be chosen such that~$R\cup A_1$ contains a \emph{tour decomposition}, that is, a decomposition in which all subgraphs are tours (see Lemma~\ref{lemma:tours}).
The second step is to transform the found tour decomposition in the remainder to a $C_{\ell}$-decomposition (see details in Section \ref{subsection:tourtoCl}).
Finally, in Section \ref{subsection:proofabsorbing} we combine both steps to prove Lemma~\ref{lemma:absorbinglemma}.

\subsection{Tour decomposition}
\label{subsection:tours}
The main goal of this subsection is to prove the following lemma.

\begin{lemma}\label{lemma:tours}
	Let $\ell \ge 7$,
	$\eps > 0$, and $n, m \in \NATS$ be such that $1/n \ll \eps, 1/m, 1/\ell$.
	Let $H$ be a $3$-graph on $n$ vertices with $\delta_2(H) \ge (2/3 + \eps)n$.
	Let $R \subseteq H$ be $C_{\ell}$-divisible on at most $m$ vertices.
	There exists a subgraph~$A_1\subseteq H$, edge-disjoint with~$R$, such that
	\begin{enumerate}[{\upshape (i)}]
        \item $A_1$ has at most $30\binom{m}{3}\ell(6\ell+1)$ edges, \item $A_1\cup R$ spans at most $30\binom{m}{3}\ell(6\ell+1)$ vertices.
	    \item $A_1$ has a~$C_\ell$-decomposition, and
	    \item $A_1\cup R$ has a tour decomposition,
	\end{enumerate}
\end{lemma}

\subsubsection{Tour-trail decompositions}

We consider decompositions~$\mathcal T = \{C_1,\dots, C_t, P_1,\dots, P_k\}$ in which $C_i$ is a tour for every~$i\in [t]$ and~$P_j$ is a trail for every~$j\in [k]$. 
In this case we say~$\mathcal T$ is a \emph{tour-trail decomposition}.
Note that every $3$-graph has a tour-trail decomposition, since we can consider every single edge in a $3$-graph as a trail on three vertices (by giving it an arbitrary ordering).

For a trail~$P =u_1 u_2 \dotsb u_{k-1}u_k$ we say that the ordered pairs~$(u_2, u_1)$ and~$(u_{k-1}, u_k)$ are the~\textit{ends} of $P$.
We denote the those pairs as~$\Ends(P)$.
Observe that the set of ends of a $P$ depends on the edge-set of $P$ only, i.e. is independent of order in which we transverse the trail.
We remark that the ends differ from the start and terminus of $P$ (as defined in Section~\ref{subsection:notation}) since they have different orderings.

Given~$H$ and a tour-trail decomposition~$\mathcal T = \{C_1, C_2,\dots, C_t, P_1, P_2, \dots, P_k\}$ of some $R \subseteq H$, we define the \emph{residual digraph of~$\mathcal T$}, denoted as~${D}(\mathcal T)$, as the multidigraph on the same vertex set as~$H$, where the arcs correspond to the union of the ordered ends of each trail of $\mathcal{T}$,
considered with repetitions.
Thus $D(\mathcal{T})$ has exactly $2t$ arcs, counted with multiplicities, if and only if $\mathcal{T}$ has $t$ trails.
For a given pair of vertices~$u,v\in V$ we denote the multiplicity of the pair $(u,v)$ in $D(\mathcal{T})$ as~$\mu_\mathcal T(u,v)$. 
Outdegrees and indegrees of a vertex $x$ in $D(\mathcal{T})$ are denoted by $\outd_{D(\mathcal{T})}(x), \ind_{D(\mathcal{T})}(x)$ respectively, omitting subscripts from the notation if the underlying digraph is clear from context.

\begin{remark}\label{rem:extendtrails}
	Observe that if~$(x,y), (y,x) \in E({D}_\mathcal T)$ then, there are two trails~$P_i$ and~$P_j$ in $\mathcal{T}$ that can be merged into a trail (if~$i\neq j)$ or tour (if~$i=j$) which contains all the edges contained in~$P_i$ and~$P_j$. 
	Thus there is another tour-trail decomposition~$\mathcal T'$ of $R$ with less trails than~$\mathcal T$, obtained from $\mathcal{T}$ by removing $P_i, P_j$ and adding the tour or trail born from joining $P_1$ and $P_2$.
\end{remark}

We construct $A_1$ in Lemma~\ref{lemma:tours} as follows.
We begin with an arbitrary tour-trail decomposition $\mathcal{T}_0$ of $R$ and we will find an increasing sequence of subgraphs $\emptyset = T_0 \subseteq T_1 \subseteq \dotsb \subseteq T_k \subseteq H$.
Each~$T_i \setminus T_{i-1}$ will be sufficiently small,~$C_\ell$-decomposable and edge-disjoint from~$T_{i-1}$. 
Moreover, each $T_i \setminus T_{i-1}$ will be a `gadget' $3$-graph of a prescribed family, which is designed to modify $T_{i-1}$ locally.
More precisely, for each $i > 0$, each $T_{i} \cup R$ will contain a tour-trail decomposition $\mathcal{T}_i$, obtained from the tour-trail decomposition $\mathcal{T}_{i-1}$ of $T_{i-1} \cup R$, and the gadget $T_i \setminus T_{i-1}$ will be chosen carefully so the residual digraph is slightly modified and becomes `simpler'.
At the end, we will have found $T_k$ and a tour-trail decomposition $\mathcal{T}_k$ of $R \cup T_k$ which has an empty residual digraph.
Thus $\mathcal{T}_k$ is actually a tour decomposition, and we finish by setting $A_1 = T_k$.

The following lemma establishes a crucial property of residual digraphs.

\begin{lemma}\label{lemma:outindeg}
	Let~$H=(V,E)$ be a~$3$-vertex-divisible hypergraph and let $\mathcal T$ be a tour-trail decomposition of $H$ with residual digraph ${D}(\mathcal T)$.
	For every $x\in V$ we have that $$\outd(x) \equiv \ind(x) \pmod 3.$$
\end{lemma}
\begin{proof}
    For every vertex $x \in V(H)$, we need to show that $\outd(x) - \ind(x) \equiv 0 \bmod 3$ in the digraph $D(\mathcal{T})$.
    Consider the auxiliary digraph $F(\mathcal{T})$ obtained as follows: for every trail or tour $P = w_1 w_2 \dotsb w_{\ell}$ in $\mathcal{T}$, to $F(\mathcal{T})$ add the arcs $(w_i, w_{i+1})$ and $(w_{i+2}, w_{i+1})$ for every $1 \leq i \leq \ell - 2$ (and for tours, add $(w_{\ell-1}, w_{\ell}), (w_1,w_{\ell}), (w_{\ell},w_1),(w_2, w_1)$ as well), including all repetitions.
    In such a way (and since $\mathcal{T}$ is a decomposition) every edge of $H$ contributes with exactly two arcs to $F(\mathcal{T})$.
    It is straightforward to check $D(\mathcal{T}) \subseteq F(\mathcal{T})$ and, crucially, that $$\outd_{D(\mathcal{T})}(x) - \ind_{D(\mathcal{T})}(x) = \outd_{F(\mathcal{T})}(x) - \ind_{F(\mathcal{T})}(x),$$ so from now on we work with $F(\mathcal{T})$ only.
    
    Let $x \in V(H)$.
    Each edge $xyz$ in $H$ contributes with two arcs to $F(\mathcal{T})$,
    which can be of type $\{ (x,y),(x,z)\}, \{ (y,x),(y,z)\}$, or $\{ (z,x),(z,y)\}$.
    The edges of the first type contribute with $2$ to $\outd(x) - \ind(x)$ in $F(\mathcal{T})$.
    The edges of second and third type contribute with $-1$ to $\outd(x) - \ind(x)$ in $F(\mathcal{T})$, which is congruent to $2 \bmod 3$.
    Thus we deduce
    $\outd(x) - \ind(x) \equiv 2 |\deg_H(x)| \bmod 3$.
    Since $H$ is $3$-vertex-divisible, this is congruent to $0 \bmod 3$, and we are done.
\end{proof}

\subsubsection{Gadgets}
In the following three lemmata we describe the aforementioned gadgets, and their main properties.

First, for a given tour-trail decomposition~$\mathcal T$ of~$R \subseteq H$ and three distinct vertices $v_1,v_2,v_3$, the following lemma states that there is a subgraph~$S_3=S_3(v_1,v_2,v_3)\subseteq H$ edge-disjoint with~$R$ and which contains a~$C_\ell$-decomposition.
Moreover, there is a tour-trail decomposition of~$R\cup S_3$ such that its residual digraph is exactly~$D(\mathcal T)$ with the additional arcs~$(v_1,v_2)$, $(v_2,v_3)$, and twice the arc~$(v_1,v_3)$.
We define the multidigraph~$\vec S_3(v_1,v_2,v_3) =\{(v_1,v_3), (v_1,v_3), (v_1,v_2), (v_2,v_3)\}$.

For two multidigraphs $D_1, D_2$, we set the notation $D_1 \sqcup D_2$ to mean the multigraph on $V(D_1) \cup V(D_2)$ obtained by adding all the arcs of $D_2$ to $D_1$, considering the multiplicities.

\begin{lemma}\label{lemma:doublestep}
    Let $\ell \ge 7$, $\eps > 0$ and $n, m \in \NATS$ be such that $1/n \ll \eps, 1/m, 1/\ell$.
	Let $H$ be a $3$-graph on $n$ vertices with $\delta_2(H) \ge (2/3 + \eps)n$.
	Given three distinct vertices $v_1,v_2,v_3\in V(H)$,~$R \subseteq H$ on at most $m$ vertices, and a tour-trail decomposition~$\mathcal T$ of~$R$ the following holds.
	There is a subgraph~$S_3=S_3(v_1,v_2, v_3)\subseteq H$, edge-disjoint from~$R$, and a tour-trail decomposition~$\mathcal T_{S_3}=\mathcal T_{S_3}(\mathcal T, v_1,v_2, v_3)$ of~$R\cup S_3$ such that
	\begin{enumerate}[label={\upshape({\roman*\,}$_{S_3}$)}]
	    \item\label{it:S3small} $S_3$ contains at most~$2\ell$ edges and $S_3 \cup R$ spans at most~$m+2\ell-3$ vertices,
	    \item \label{it:CldecompositionS3}$S_3$ has a~$C_\ell$-decomposition, and
	    \item\label{it:arrowsS3} ${D}(\mathcal T_{S_3}) = {D} (\mathcal T) \, \sqcup \,\vec S_3(v_1,v_2,v_3)$.
	\end{enumerate}
\end{lemma}

\begin{proof}
     The minimum codegree condition on~$H$ implies that there is a vertex~$x\in V(H)$ that lies in~$N(v_1v_2)\cap N(v_1v_3)\cap N(v_2v_3)$.
    Considering the paths~$v_1v_3x$ and~$v_3xv_2v_1$, two applications of Lemma~\ref{lemma:therearemanycycles} yield the existence of two edge-disjoint cycles~$C_1$ and~$C_2$ of length~$\ell$, edge-disjoint with~$R$, and such that $v_1v_3x\in E(C_1)$ and~$v_3xv_2, xv_2v_1\subseteq E(C_2)$ (transversing the vertices in that order).
    Then~$S_3 = C_1\cup C_2$, clearly satisfies \ref{it:S3small} and \ref{it:CldecompositionS3}.
    Hence, we only need to prove the existence of a tour-trail decomposition~$\mathcal T_{S_3}$ of $R \cup S_3$ for which~\ref{it:arrowsS3} holds. 
    
    For this, consider the trail~$P_1 = v_3v_2xv_1v_3$.
    Observe that~$E(S_3) \setminus E(P_1)$ consists exactly in the edges of a trail $P_2$ whose ends are~$(v_1,v_2)$ and~$(v_1,v_3)$. 
    Indeed, the edges contained in the set~$E(C_2)\setminus \{v_3v_2x, v_2xv_1\}$ form a trail between~$(v_2,v_1)$ and~$(v_3,x)$, that we may merge with the trail with edges in~$E(C_1)\setminus \{xv_1v_3\}$ from~$(v_3,x)$ to $(v_1,v_3)$. 
    Therefore, $\mathcal T_{S_3} = \mathcal T \cup \{P_1, P_2\}$ is a tour-trail decomposition of~$R\cup S_3$.
    We deduce \ref{it:arrowsS3} by noticing that the ends of~$P_1$ and~$P_2$ are~$(v_2, v_3)$ and $(v_1,v_3)$, and~$(v_1,v_2)$ and~$(v_1,v_3)$ respectively.
\end{proof}

The following is our second gadget.
It is designed so we can add a small subgraph $C_4 \subseteq H$ to some $R$, such that $R \cup C_4$ has a tour-trail decomposition in which the residual digraph has an extra directed four-cycle.
We use the notation~$\vec C_4(v_1, v_2, v_3, v_4)=\{(v_1,v_2),(v_2,v_3), (v_3,v_4),(v_4,v_1)\}$.

\begin{lemma}\label{lemma:4cycle}
    Let $\ell \ge 7$, $\eps > 0$ and $n, m \in \NATS$ such that $1/n \ll \eps, 1/m, 1/\ell$.
	Let $H$ be a $3$-graph on $n$ vertices with $\delta_2(H) \ge (2/3 + \eps)n$.
	Given four distinct vertices $v_1,v_2,v_3,v_4\in V(H)$, a subgraph~$R \subseteq H$ on at most $m$ vertices, and a tour-trail decomposition~$\mathcal T$ of~$R$ the following holds. 
	There is a subgraph~$C_4=C_4(v_1,v_2, v_3,v_4)\subseteq H$, edge-disjoint from~$R$ and a tour-trail decomposition~$\mathcal T_{C_4}=\mathcal T_{C_4}(\mathcal T,v_1,v_2, v_3,v_4)$ of~$R\cup C_4$ such that
	\begin{enumerate}[label={\upshape({\roman*\,}$_{C_4}$)}]
	    \item\label{it:C4small} $C_4$ has at most~$8\ell$ edges and~$C_4 \cup R$ spans at most~$m+4\ell-6$ vertices,
	    \item \label{it:CldecompositionC4}$C_4$ has a~$C_\ell$-decomposition, and
	    \item\label{it:arrowsC4} ${D}(\mathcal T_{C_4}) = {D} (\mathcal T) \, \sqcup \,\vec C_4(v_1,v_2,v_3,v_4)$.
	\end{enumerate}
\end{lemma}

\begin{proof}
    Two consecutive applications of Lemma~\ref{lemma:doublestep} yield the existence of edge-disjoint subgraphs $S_3(v_1,v_2,v_3)$ and~$S_3(v_3,v_1,v_4)$. 
    More precisely, first we apply Lemma~\ref{lemma:doublestep} to obtain~$S_3(v_1,v_2,v_3)$ edge-disjoint from~$R$. 
    Then, we apply it again with~$R\cup S_3(v_1,v_2,v_3)$ in place of~$R$ to obtain~$S_3(v_3,v_2,v_4)$ edge disjoint from~$R\cup S_3(v_1,v_2,v_3)$ (here we use~$1/n \ll 1/m$, to apply Lemma~\ref{lemma:doublestep} to a larger subgraph with at most~$m+2\ell-6$ vertices). 
    It is not difficult to check that the subgraph~$C_4=S_3(v_1,v_2,v_3) \cup S_3(v_3,v_1,v_4)$ satisfies~\ref{it:C4small} and~\ref{it:CldecompositionC4}
    
    Moreover, in the second application of Lemma~\ref{lemma:doublestep} we obtain a tour-trail decomposition $\mathcal{T}'$ of~$R\cup C_4$ equal to $\mathcal T'=\mathcal T_{S_3}\big(\mathcal{T}_{S_3}(\mathcal T,v_1,v_2,v_3), v_3,v_1, v_4\big),$
    whose residual digraph is given by
    $${D}(\mathcal T') = {D} (\mathcal T) \sqcup \vec S_3(v_1,v_2,v_3) \sqcup \vec S_3(v_3,v_4,v_1).$$
    Observe that~${D}(\mathcal T')$ contains both the arcs~$(v_1,v_3)$ and~$(v_3,v_1)$ twice.
    By Remark~\ref{rem:extendtrails}, we can obtain a tour-trail decomposition~$\mathcal T_{C_4}$ which satisfies \ref{it:arrowsC4}.
\end{proof}

Our third and final gadget will add the arcs of two vertex-disjoint oriented triangles to the residual digraph.
Set the notation~$\vec T_3(v_1,v_2,v_3)=\{(v_1,v_2), (v_2,v_3), (v_3,v_1)\}$ for the oriented triangle on vertices~$v_1,v_2,v_3$. 
Given six distinct vertices~$v_1,v_2,v_3,v_4,v_5,v_6$, as a final result we wish for a residual digraph consisting of the two oriented triangles~$\vec T_3(v_1,v_2,v_3)$ and~$\vec T_3(v_4,v_5,v_6)$.

This can be done using the oriented $4$-cycles of Lemma~\ref{lemma:4cycle} three times, by considering the oriented $4$-cycles $\vec C_4(v_1,v_2,v_5,v_6)$, $\vec C_4(v_2,v_3,v_4,v_5)$, and~$\vec C_4(v_1,v_6, v_4, v_3)$.
This can be thought geometrically, as the oriented~$4$-cycles forming the faces of a triangular prism, whose bases lie in the desired triangles. 
The arcs between the vertices of the two triangles will go in opposite directions, and therefore we will be able to ``cancel'' them.

To have an analogous notation as for the other two gadgets, set~$$\vec P_6(v_1,v_2,v_3,v_4,v_5,v_6)=\vec T_3(v_1,v_2,v_3)\sqcup \vec T_3(v_4,v_5,v_6).$$ 

\begin{lemma}\label{lemma:prism}
    Let $\ell \ge 7$, $\eps > 0$ and $n, m \in \NATS$ be such that $1/n \ll \eps, 1/m, 1/\ell$.
	Let $H$ be a $3$-graph on $n$ vertices with $\delta_2(H) \ge (2/3 + \eps)n$.
	Given six distinct vertices $v_1,v_2,v_3,v_4,v_5, v_6\in V(H)$,~$R \subseteq H$ on at most $m$ vertices, and a tour-trail decomposition~$\mathcal T$ of~$R$ the following holds.
	There exists a subgraph~$P_6=P_6(v_1,v_2, v_3,v_4,v_5,v_6)\subseteq H$, edge-disjoint from~$R$ and a tour-trail decomposition~$\mathcal T_{P_6}=\mathcal T_{P_6}(\mathcal T,v_1,v_2, v_3,v_4,v_5,v_6)$ of~$R\cup P_6$ such that
	\begin{enumerate}[label={\upshape({\roman*\,}$_{P_6}$)}]
	    \item\label{it:P6small} $P_6$ has at most~$12\ell$ edges and $P_6\cup R$ spans at most~$m+12\ell-18$ vertices, 
	    \item \label{it:CldecompositionP6}$P_6$ has a~$C_\ell$-decomposition, and
	    \item\label{it:arrowsP6} ${D}(\mathcal T_{P_6}) = {D} (\mathcal T) \sqcup \vec P_6(v_1,v_2,v_3,v_4,v_5,v_6)$
	\end{enumerate}
\end{lemma}

\begin{proof}
    Using $1/n \ll 1/m$ we apply Lemma~\ref{lemma:4cycle} iteratively three times, to obtain three edge-disjoint subgraphs $C_4(v_1, v_2, v_5, v_6)$, $C_4(v_2, v_3, v_4, v_5)$, and $C_4(v_1, v_6, v_4, v_3)$, which are also edge-disjoint from $R$.
    It is straightforward to check that $P_6=C_4(v_1,v_2, v_5, v_6)\cup C_4(v_2,v_3,v_4,v_5)\cup C_4(v_1,v_6,v_4,v_3)$
    satisfies \ref{it:P6small} and \ref{it:CldecompositionP6}.
    
    The last application of Lemma~\ref{lemma:4cycle} yields a tour-trail decomposition~$\mathcal T'$ of~$R\cup P_6$ 
    with residual digraph given by 
    $${D}(\mathcal T') 
    =
    {D}(\mathcal T) \sqcup 
    \vec C_4(v_1,v_2,v_5,v_6) \sqcup
    \vec C_4(v_2,v_3,v_4,v_5) \sqcup 
    \vec C_4(v_1,v_6,v_4,v_3).$$
    ${D}(\mathcal T')$ contains the arcs~$(v_1,v_6)$, $(v_6,v_1)$, $(v_2,v_5)$, $(v_5,v_2)$, $(v_3,v_4)$, and $(v_4,v_3)$, and by Remark~\ref{rem:extendtrails} we can remove them to obtain a tour-trail~$\mathcal T_{P_6}$ which satisfies~\ref{it:arrowsP6}.
\end{proof}

\subsubsection{The sea of triangles}
In what follows, we will use the previous gadgets to find, for any given~$R\subseteq H$, an edge-disjoint small $C_\ell$-decomposable~$T \subseteq H$,
the main property being that~$R\cup T$ contains a tour-trail decomposition with residual digraph consisting only of vertex-disjoint oriented triangles.

The following definitions will be useful for this propose. 
Given a multidigraph~$D=(V, E)$, a \emph{triangle lake} $T \subseteq D$ is an induced subdigraph with vertices in~$V'\subseteq V$ that consists only of vertex-disjoint (simple) oriented triangles and such that there is no arc between~$V'$ and~$V\setminus V'$ or vice versa.
Any $D$ contains a unique vertex-maximal triangle lake (possibly empty), we call such subdigraph the \emph{sea of triangles of~$D$} and we denote it by~$\vec\triangle(D)$.
If~$D=\vec\triangle (D)$ we say~$D$ is itself a sea of triangles.
Given two directed digraphs~$D_1$ and~$D_2$ on the same vertex set, we establish the notation~$D_1-D_2$ to mean the multigraph resulting from subtracting the edges of~$D_2$ from~$D_1$ counting the multiplicities. 

As for hypergraphs, we do not distinguish between the directed multigraph~$D=(V,E)$ and the set of arcs~$E$.

\begin{lemma}\label{lemma:seaoftriangles}
	Let $\ell \ge 7$,
	and $\eps > 0$ and $n, m \in \NATS$ be such that $1/n \ll \eps, 1/m, 1/\ell$.
	Let $H$ be a $3$-graph on $n$ vertices with $\delta_2(H) \ge (2/3 + \eps)n$.
	Let $R \subseteq H$ be $C_{\ell}$-divisible on at most $m$ vertices.
	There exists a subgraph~$T\subseteq H$, edge-disjoint from~$R$, such that
	\begin{enumerate}[label={\upshape({\roman*\,}$_{\ref{lemma:seaoftriangles}}$)}]
	    \item\label{it:smalledgesTS} $T$ has at most~$30\binom{m}{3}\ell$ edges,
	    \item\label{it:smallvertTS}$T\cup R$ spans at most $30\binom{m}{3}\ell$ vertices,
	    \item\label{it:CldecompositionTS} $T$ has a~$C_\ell$-decomposition, and
	    \item\label{it:arrowsTS} there is a tour-trail decomposition~$\mathcal T_\triangle$ of $T\cup R$ such that $D(\mathcal{T}_\triangle)$ is a sea of triangles.
	\end{enumerate}
\end{lemma} 

\begin{proof}
    Set~$k=2\binom{m}{3}+1$. 
    To find $T$, we will iteratively find subgraphs~$T_i\subseteq H$ for every~$0 \leq i \leq k$ such that each $T_i$ has a~$C_\ell$-decomposition, is edge-disjoint with respect to~$R$, contains at most~$14\ell i$ edges,
    and such that~$T_i\cup R$ spans at most~$m+i(14\ell-38)$ vertices.
    We will see that the last subgraph~$T_k$ satisfies the desired properties.
    Note that in this case properties~\ref{it:smalledgesTS}, \ref{it:smallvertTS}, and \ref{it:CldecompositionTS} would follow directly since~$14\ell i$ and $m+i(14\ell-38)$ are smaller than $30\binom{m}{3}\ell$ for every $i\leq k$.
    Hence, most or our effort is dedicated to ensure~\ref{it:arrowsTS}.
    To do so, at each step we will define a tour-trail decomposition~$\mathcal T_i$ of~$T_i\cup R$ such that its residual digraph will be almost identical to the one of~$\mathcal T_{i-1}$ except for a few subtly chosen arcs.
    Additionally, we will define auxiliary vertex sets~$X_i$ of size at least~$n/2-4i$ such that~$(T_i\cup R)[X_i]$ is empty. 
    
    Since~$1/n \ll 1/m, 1/\ell$ and~$T_i\cup R$ spans at most~$30\binom{m}{3}\ell$ vertices for every~$i\in [k]$, $n$ will be sufficiently large to apply lemmata~\ref{lemma:doublestep}--\ref{lemma:prism} with~$T_i\cup R$ in place of $R$, and we will do this without further comment.
     
    For $i = 0$, take~$T_0= \emptyset$ and~$\mathcal T_0$ to be an arbitrary tour-trail decomposition of~$R$ (this always exists).
    Also, let $X_0\subseteq V(H)$ have size~$\lceil n/2 \rceil$ such that~$R[X_1]$ is empty, which can be done since $1/n \ll 1/m$.
    Now, for~$0 \leq i<k$, given~$T_i$,~$\mathcal T_i$ and~$X_i$ define~$T_{i+1}$,~$\mathcal T_{i+1}$ and~$X_{i+1}$ using the following set of rules:
    
    \begin{enumerate}[label={\upshape (\Roman*)}]
        \item \label{case:opparrows} Suppose there are vertices~$a,b\in V$ such that~$(a,b), (b,a)\in D(\mathcal T_i)$. 
        In this case just set~$T_{i+1}= T_i$ and~$X_{i+1}=X_i$, and let $\mathcal{T}_{i+1}$ be a tour-trail decomposition such that
        \begin{align}\label{case:opparrowdigraph}
            D(\mathcal T_{i+1}) = D(\mathcal T_i) - \{(a,b),(b,a)\},
        \end{align}
        which exists by Remark~\ref{rem:extendtrails}.
        
        \item \label{case:multiedge} Suppose that~\ref{case:opparrows} does not hold and $D(\mathcal{T}_i)$ contains an arc with multiplicity more than one, i.e. there are vertices~$a,b\in V(H)$ with~$\mu_{\mathcal T_i}(a,b)>1$.
        Take~$x\in X_i$ and apply Lemma~\ref{lemma:doublestep} to~$R\cup T_i$ on the vertices~$b,x,a$ to obtain the subgraph~$S_3(b,x,a)\subseteq H$ and the tour-trail decomposition~$\mathcal T_i'=\mathcal T_{S_2}(\mathcal T_i, b,x,a)$. 
        Further, take new vertices~$y,z,w\in X_i\setminus \{x\}$ and apply Lemma~\ref{lemma:prism} on~$T_i\cup R\cup S_3(b,x,a)$ to obtain~$P_6(a,x,b,y,z,w)$ and a tour-trail decomposition~$\mathcal T_i''=\mathcal T_{P_6}(\mathcal T_i',a,x,b,y,z,w)$.
        Set~
        \begin{align*}
            T_{i+1}&= T_i\cup S_3(b,x,a)\cup P_6(a,x,b,y,z,w), \, \text{and} \\
            X_{i+1}&=X_i\setminus \{x,y,z,w\},
        \end{align*}
        and observe that~\ref{it:S3small} and~\ref{it:P6small} in Lemmata~\ref{lemma:doublestep} and~\ref{lemma:prism}, we have that~$T_{i+1}$ has at most~$14\ell i + 2\ell + 12\ell = 14\ell(i+1)$ edges and the subgraph~$T_{i+1}\cup R$ spans at most~
        $$m+i(14\ell - 21) + 14\ell -21 = m + (i+1)(14\ell - 21)$$ vertices.
        Moreover, since~$T_{i+1}$ is edge-disjoint union of subgraphs that contain~$C_\ell$-decomposition it also contains one. 
        Additionally~$\vert X_{i+1}\vert = \vert X_i \vert -4$.
        
        Observe that the resulting tour-trail decomposition~$\mathcal T_i''$ has a residual digraph given by $D(\mathcal T_i'') = {D}(\mathcal T_i)\sqcup \vec S_3(b,x,a) \sqcup \vec P_6(a,x,b,y,z,w).$
        Recall that the multiplicity of~$(a,b)$ is at least two in $D(\mathcal{T}_i)$ and, using Remark~\ref{rem:extendtrails} to annihilate edges which have opposite directions, we obtain a tour-trail decomposition~$\mathcal T_{i+1}$ such that
        \begin{align}\label{case:multiedgedigraph}
            {D}(\mathcal T_{i+1}) = {D}(\mathcal T_i)\setminus \{(a,b), (a,b)\} \sqcup \{(b,a)\} \sqcup \vec T_3(y,z,w).
        \end{align}
        
        \item \label{case:path} Suppose cases~\ref{case:opparrows} and~\ref{case:multiedge} do not hold, and that there are three distinct vertices~$a,b,c\in V(H)$ such that~$(a,b), (b,c)\in {D}(\mathcal T_i)\setminus \vec\triangle({D}(\mathcal T_i))$.
        
        Consider vertices~$x,y,z\in X_i$ and apply Lemma~\ref{lemma:prism} on the vertices~$c,b,a,x,y,z$ to obtain~$P_6(c,b,a,x,y,z)$ and the tour-trail decomposition~$\mathcal T_i'=\mathcal T_{P_6}(\mathcal T_i,c,b,a,x,y,z)$.
        Setting
        \begin{align*}
        T_{i+1}= T_i\cup P_6(c,b,a,x,y,z) \quad \text{ and } \quad X_{i+1}=X_i\setminus \{x,y,z\},
        \end{align*}
        we deduce 
        from Lemma~\ref{lemma:prism} that~$T_{i+1}$ has at most~$14\ell i+12\ell\leq 14\ell(i+1)$ edges and that~$T_{i+1}\cup R$ spans at most~$m+i(14\ell-38)+12\ell - 18 \leq m+(i+1)(14\ell - 21)$ vertices. 
        Moreover, it is clear that~$T_{i+1}$ contains a~$C_\ell$-decomposition and that~$\vert X_{i+1}\vert\geq \vert X_i\vert -4$.

        The residual digraph of~$\mathcal T_i'$ 
        is ${D}(\mathcal T_i')= {D}(\mathcal T_i) \sqcup \vec P_6(c,b,a,x,y,z)$ and therefore, using Remark~\ref{rem:extendtrails}, we obtain a tour-trail decomposition~$\mathcal T_{i+1}$ such that
        \begin{align}\label{case:pathdigraph}
            {D}(\mathcal T_{i+1}) = ({D}(\mathcal T_i)- \{(a,b),(b,c)\}) \sqcup  \{(a,c)\}  \sqcup\vec T_3(x,y,z).
        \end{align}
    
        \item \label{case:star} Suppose that cases~\ref{case:opparrows},~\ref{case:multiedge} and~\ref{case:path} do not hold, and that there are vertices~$a,b,c,d$ such that~$(a,b),(a,c),(a,d)\in{D}(\mathcal T_i)$. 
        Apply Lemma~\ref{lemma:doublestep} on the vertices~$c,d,a$ to obtain~$S_3(c,d,a)$ and a tour-trail decomposition~$\mathcal T_i'=\mathcal T_{S_3}(\mathcal T_i, c,d,a)$. 
        Further, take~$x,y,z\in X_i$ apply Lemma~\ref{lemma:prism} to~$T_i\cup R\cup S_3(c,d,a)$ on the vertices~$a,c,b,x,y,z$ to obtain $P_6(a,c,b,x,y,z)$ and the tour-trail decomposition~$\mathcal T_i''=\mathcal T_{P_6}(\mathcal T_i',a,c,b,x,y,z)$. 
        Set 
        \begin{align}\label{case:stargraph}
            \begin{split}
             T_{i+1}&= T_i\cup S_3(c,d,a) \cup P_6(a,c,b,x,y,z) \, \text{and} \\
             X_{i+1}&=X_i\setminus \{x,y,z\},     
            \end{split}
        \end{align}
        and observe that~$T_{i+1}$ has at most~$14\ell(i+1)$ edges and that~$T_{i+1}\cup R$ spans at most~$m+i(14\ell - 21)+14\ell -21 = m+(i+1)(14\ell - 21)$. 
        Again, it is easy to check that~$T_{i+1}$ contains a~$C_\ell$-decomposition and that~$\vert X_{i+1}\vert \geq \vert X_i\vert -4$.
        
        Observe that the residual digraph of the tour-trail decomposition~$\mathcal T_i''$ is given by $${D}(\mathcal T_i'') = {D}( \mathcal T_i) \sqcup \vec S_3(c,d,a) \sqcup \vec P_6(a,c,b,x,y,z),$$
        and again, by Remark~\ref{rem:extendtrails} we can find a tour-trail decomposition~$\mathcal T_{i+1}$ such that 
        \begin{align}\label{case:stardigraph}
            {D}(\mathcal T_{i+1}) = ({D}(\mathcal T_i) - \{(a,b),(a,c), (a,d)\}) \sqcup  \{(c,b), (c,d)\}  \sqcup\vec T_3(x,y,z).
        \end{align}
    \item \label{case:end} If none of the previous cases takes place, then set~$T_{i+1}=T_i$ and~$\mathcal T_{i+1} = \mathcal T_i$.
    \end{enumerate}
    
    Let $T = T_k$ and $\mathcal{T}_\triangle = \mathcal{T}_k$.
    As discussed before, we have ensured \ref{it:smalledgesTS}--\ref{it:CldecompositionTS} hold by construction.
    To prove \ref{it:arrowsTS} we have to show all arcs of~${D}(\mathcal T)$ are in its sea of triangles $\vec\triangle({D}(\mathcal T_k))$.
    We shall require the following definition.
    For any tour-trail decomposition~$\mathcal T$ define the parameter~$\Phi(\mathcal T) = \vert E({D}(\mathcal T)) \vert - \vert E(\vec \triangle({D}(\mathcal T)))\vert$.
    In words, $\Phi(\mathcal T)$ is the number of arcs in $\mathcal{T}$ which are not in its sea of triangles.
    Note $\Phi(\mathcal T) \ge 0$ always.
    
    First, we claim that there exists some $0 \leq i < k$ such that case~\ref{case:end} happens when processing $\mathcal{T}_i$.
    Suppose this is not the case.
    Observe that if any of the cases~\ref{case:opparrows}--\ref{case:star} happens when processing $\mathcal{T}_i$, due to the structure of~${D}(\mathcal T_{i+1})$ given in~\ref{case:opparrows}, \eqref{case:multiedgedigraph}, \eqref{case:pathdigraph}, and \eqref{case:stardigraph}, we have 
    $$\Phi(\mathcal T_{i+1}) \leq \Phi(\mathcal T_{i})-1.$$
    Hence, we have $\Phi(\mathcal T_{k}) \leq \Phi(\mathcal{T}_0) - k$.
    Note that the number of arcs in $D(\mathcal T_0)$ is twice the number of trails of $\mathcal{T}_0$,
    each trail uses at least one edge of $R$,
    and $R$ has at most $\binom{m}{3}$ edges since it spans at most $m$ vertices.
    Thus $\Phi(\mathcal{T}_0) \leq 2 |E(R)| \leq 2\binom{m}{3}$.
    Since~$\Phi(\mathcal T_k)\geq 0$, we deduce $2 \binom{m}{3} < k \leq \Phi(\mathcal{T}_0) \leq 2 \binom{m}{3}$,
    a contradiction.
    This proves the claim.
    
    Now, let $\vec G={D}(\mathcal T_\triangle) \setminus\vec \triangle({D}(\mathcal T_\triangle))$.
    To prove~\ref{it:arrowsTS} we need to show $\vec G$ is empty.
    Note that once the procedure falls in~\ref{case:end} in a step $i < k - 1$, it will happen again in step $i+1$.
    Therefore, by the previous discussion, we know that case~\ref{case:end} happened when processing $\mathcal{T}_{k-1}$ to build $\mathcal{T}_k$.
    In particular, $\mathcal{T}_{k-1} = \mathcal{T}_k = \mathcal{T}_\triangle$ and we know cases~\ref{case:opparrows}--\ref{case:star} did not hold when processing $\mathcal{T}_{k-1}$.
    
    Denote the vertices spanned by the arcs of~$\vec G$ as~$V$.
    Observe first that there are no vertices~$a,b\in V$ such that~$(a,b),(b,a)\in D(\mathcal T_{k-1})$ otherwise case \ref{case:opparrows} would have hold.
    Then, notice that for every pair~$a,b\in V$ we have that~$\mu_{\mathcal T_{k-1}}(a,b) \leq 1$, otherwise $\mathcal T_{k-1}$ would have qualified for case~\ref{case:multiedge}.
    This implies that $\vec G$ is an oriented graph, with no multiple edges or directed $2$-cycles.
    Moreover, for every vertex $b\in V$ we have either~$\outd(b)=0$ or~$\ind(b)=0$ in $\vec G$, otherwise the case \ref{case:path} would have taken place.
    If~$\vec G$ is non-empty, then there is a vertex~$b\in V$ with~$\outd(b) > 0$, which then implies~$\ind(b)=0$.
    Then Lemma~\ref{lemma:outindeg} implies that~$\outd(b)\geq 3$.
    Therefore, $\mathcal{T}_{k-1}$ would have fallen in case~\ref{case:star}, a contradiction.
    Thus $\vec G$ is empty, which finally shows~\ref{it:arrowsTS}.
\end{proof}

Now we are ready to prove the main lemma of this subsection.

\begin{proof}[Proof of Lemma~\ref{lemma:tours}]
    Let~$T\subseteq H$ as found in Lemma~\ref{lemma:seaoftriangles} and let~$\mathcal T_\triangle$ be a tour-trail decomposition of~$R\cup T$ given in \ref{it:arrowsTS} such that it residual digraph is a sea of triangles.
    
    Since each trail of~$\mathcal T_\triangle$ contributes two arcs to~$\mathcal {D}(\mathcal {T}_\triangle)$ the number of arcs is even, and so is the number of oriented triangles in $\mathcal {D}(\mathcal {T}_\triangle)$.
    Suppose the number of triangles is~$2k$ and let~${D}(\mathcal T_\triangle))= \bigcup_{i\in [2k]} \vec T_3(a_i,b_i,c_i)$. 
    Since the triangles are vertex-disjoint and by \ref{it:smallvertTS} we have that~$2k\leq 30\binom{m}{3}\ell$.
    
    Apply Lemma~\ref{lemma:prism} to obtain the prism $P_1 = P_6(c_{1},b_{1}, a_{1}, c_{2},b_{2}, a_{2})$ and the tour-trail decomposition~$\mathcal T'$ of $R \cup T \cup P_1$ whose residual digraph is given by
    \begin{align*}
        D(\mathcal T') &=  \vec T_3(c_{1},b_{1}, a_{1})\sqcup \vec T_3(c_{2},b_{2}, a_{2}) \sqcup \bigcup_{i = 1}^{2k} \vec T_3(a_i,b_i,c_i) \\
        &= 
        \vec T_3(c_1,b_1,a_1)\sqcup
        \vec T_3(a_1,b_1,c_1) \sqcup 
        \vec T_3(a_2,b_2,c_2)  \sqcup
        \vec T_3(c_{2},b_{2}, a_{2})\sqcup 
        \bigcup_{i = 3}^{2k} \vec T_3(a_i,b_i,c_i).
    \end{align*}
    Using Remark~\ref{rem:extendtrails} we can ``cancel out'' the arcs of triangles $\vec T_3(c_1,b_1,a_1)$, $\vec T_3(a_1,b_1,c_1)$, $\vec T_3(a_2,b_2,c_2)$, and $\vec T_3(c_{2},b_{2}, a_{2})$, and obtain a tour-trail decomposition $\mathcal{T}_i$ whose residual digraph is a sea of triangles with~$2k-2$ triangles. 
    
    Since~$1/n\ll1/m$, and every prism spans at most~$12\ell-18$ new vertices, we may assume that~$n$ is large enough for~$k-1\leq 15\binom{m}{3}\ell - 1$ extra applications of Lemma~\ref{lemma:prism}, adding the prism $P_i = P(c_{2i-1},b_{2i-1},a_{2i-1},c_{2i},b_{2i},a_{2i})$ in each step $2 \leq i \leq k$.
    Therefore, we can repeat the previous argument until there are no more triangles in the residual digraph (and hence, no more arcs).
    Taking~$A_1 = T\cup \bigcup_{i\in [k]} P_i$, it is easy to check that it satisfies all the desired properties.  
\end{proof}

\subsection{From a tour decomposition to a cycle decomposition}
\label{subsection:tourtoCl}

In this section we prove the following lemma, which constructs an absorber given a $C_{\ell}$-divisible remainder which has a tour decomposition.

\begin{lemma}\label{lemma:tourstocycles}
    Let~$\ell\geq 7$,~$\eps>0$, and~$n,m\in \NATS$ be such that~$1/n\ll \eps, 1/m, 1/\ell$. 
    Let~$H$ be a~$3$-graph on~$n$ vertices with~$\delta_2(H)\geq (2/3 + \eps)n$.
    Let~$R\subseteq H$ be a~$C_\ell$-divisible edge-disjoint collection of tours spanning at most~$m$ vertices in total.
    Then, there is a $C_{\ell}$-absorber~$A_2$ for~$R$, such that~$A_2\cup R$ spans at most~$10\binom{m}{3}\ell^2$ edges. 
\end{lemma}

Given two subgraphs~$R_1$ and~$R_2$, we say that a subgraph~$T\subseteq H$ edge-disjoint from~$R_1$ and~$R_2$ is a \emph{$(R_1,R_2)$-transformer} if $T[V(R_1)], T[V(R_2)]$ are empty and both~$T\cup R_1$ and~$T\cup R_2$ contain a~$C_\ell$-decomposition. 
Observe that if~$R_2$ has a~$C_\ell$-decomposition, then $T\cup R_2$ is an absorber for~$R_1$.

\begin{lemma}\label{lemma:ladders}
    Let~$\ell\geq 7$,~$\eps>0$, and~$n,m\in \NATS$ be such that~$1/n\ll \eps, 1/m, 1/\ell$. 
    Let~$H$ be a~$3$-graph on~$n$ vertices with~$\delta_2(H)\geq (2/3 + \eps)n$.
    Let~$R\subseteq H$ be a tour and~$C\subseteq H$ be a cycle.
    Suppose that~$R$ and~$C$ are edge-disjoint and contain the same number of edges, which is at most~$m$.
    Then~$H$ contains an~$(R,C)$-transformer~$L$ with at most~$m\ell$ edges and spanning at most~$m(\ell-4)$ vertices.
\end{lemma}

\begin{proof}
Let~$r_1, r_2, \dots, r_m$ and~$c_1, c_2,\dots, c_m$ the sequence of vertices of~$R$ and~$C$ respectively (recall that while $C$ does not contain repetitions,~$R$ may contain). 

In the following, all operations on the indices are modulo~$m$.
We define iteratively the following paths~$P_i, Q_i$ for every~$i\in [m]$. 
Apply Lemma~\ref{lemma:therearemanycycles} to obtain a path $P_i$ on~$5$ vertices, edge-disjoint from~$R\cup C$, from the pair~$(r_i, r_{i+1})$ to the pair~$(c_{i-1}, c_{i})$.
Similarly, we can obtain a path~$Q_i$ on~$\ell-5$ vertices, from the pair~$(r_{i},r_{i-1})$ to the pair~$(c_i, c_{i-1})$, edge disjoint from~$R\cup C$, and with no interior vertex in common with the paths~$P_i$, $P_{i-1}$.

We claim that~$L=\bigcup_{i\in [m]} \left(P_i\cup Q_i\right)$ is the desired transformer.
Indeed, observe that the edges of~$P_i$ and~$Q_i$ together with the edge~$r_{i-1}r_{i}r_{i+1}\in E(R)$ form a cycle of length~$\ell$, thus $R \cup L$ can be decomposed into those $\ell$-cycles.
In the same way, the edges of~$P_{i-1}$ and~$Q_i$ together with the edge~$c_{i-2}c_{i-1}c_{i}\in E(C)$ form a cycle of length~$\ell$, and therefore all those cycles form a $C_{\ell}$-decomposition of $C \cup L$.
\end{proof}

For any~$k,\ell \in \NATS$ we define~$B(k,\ell)$ to be the~$3$-graph resulting from a cycle of length~$k\ell$ with vertices in~$\{v_1, v_2,\dots, v_{k\ell}\}$ and identifying all vertices~$v_i$ with $i\equiv 1 \bmod \ell$ and all vertices~$v_j$ with~$j\equiv 2 \bmod \ell$. 
This is to say that~$B(k,\ell)$ consists of~$k$ copies of cycles of length~$\ell$ glued through exactly two vertices, and those two vertices are consecutive in every cycle. 
Observe that~$B(k,\ell)$ is a tour and admits a~$C_\ell$-decomposition.

Now we are ready to prove Lemma~\ref{lemma:tourstocycles}.

\begin{proof}[Proof of Lemma~\ref{lemma:tourstocycles}]
    Consider the tours $T_1, T_2, \dots, T_k$ in~$R$ and observe that~$k\leq \binom{m}{3}/4$ (each tour has at least~$4$ edges). 
    First, we want to reduce the proof to the case in which there is a single \href{https://www.youtube.com/watch?v=4oOWghSh3_Q}{\color{black}{long tour}}. %
    Suppose~$k\geq 2$ and take~$a_i, b_i$ two consecutive vertices in~$T_i$ for~$i=\{1,2\}$.
    We can apply Lemma~\ref{lemma:therearemanycycles} to find a path~$P_1$ on~$5$ vertices with ends~$(b_1,a_1)$ and~$(a_2,b_2)$ which is edge-disjoint to~$R$. 
    Similarly, we can find~$P_2$ on~$\ell-5$ vertices with ends~$(a_1,b_1)$ and~$(b_2,a_2)$, edge-disjoint with~$R$, and sharing no interior vertex with~$P_1$. 
    Starting in~$(a_1,b_2)$ and then traversing sequentially $T_1$, $P_1$, $T_2$, and~$P_2$, one can check that~$T_1\cup T_2\cup P_1 \cup P_2$ forms a tour spanning at most~$\vert V(T_1\cup T_2) \vert+\ell-4$ vertices.
    Moreover, it is easy to see that $P_1\cup P_2$ is a cycle of length~$\ell$.
    By repeating this argument we can obtain~$A'\subseteq H$ edge-disjoint from~$R$, $C_{\ell}$-decomposable, and such that~$R'=R\cup A'$ consists of a single tour spanning at most~$m+k(\ell-4)$ vertices.
    Observe that since~$R$ is~$C_\ell$-divisible, then so is~$R'$.
    Let~$m'$ be the number of edges in~$R'$ and notice that
    $$m'\leq \binom{m}{3}+k\ell\leq 2\binom{m}{3}\ell$$
    
    Second, observe that by several applications of Lemma~\ref{lemma:therearemanycycles} we can find two edge-disjoint subgraphs~$B,C \subseteq H$, vertex-disjoint to each other, both of them edge-disjoint with $R'$, and such that~$B$ is a copy of~$B(m'/\ell,\ell)$ and~$C$ is a cycle of length~$m'$ (observe that~$\ell$ divides $m'$ since~$R'$ is~$C_\ell$-divisible).
    
    Now two suitable applications of Lemma~\ref{lemma:ladders} yield the result. 
    More precisely, first apply Lemma~\ref{lemma:ladders} with $R'$ in the r\^ole of~$R$ to obtain~ a $(R',C)$-transformer $L_1\subseteq H$ with at most~$m'\ell$ edges. 
    For the second application of Lemma~\ref{lemma:ladders} observe that, since~$R'\cup L_1$ contain at most~$m'(\ell+1)$ we may assume~$n$ is large enough so that~$\delta_2(H\setminus (R'\cup L_1))\geq (2/3+\eps/2)n$. 
    Hence, another application of Lemma~\ref{lemma:ladders} now with~$B$ in the r\^ole of~$R$ and~$H\setminus (R'\cup L_1)$ in the r\^ole of~$H$ yields the existence of a~$(B,C)$-transformer~$L_2\subseteq H$ edge disjoint with~$R'\cup L_1$. 

    Putting all this together, and recalling that both~$A'$ and $B$ contain a~$C_\ell$-decomposition, we have that the hypergraphs
    $$R\cup A'\cup L_1\cup C\cup L_2\cup B \quad \text{and} \quad  A'\cup L_1\cup C\cup L_2\cup B $$  
    contain~$C_\ell$-decompositions. 
    To finish the proof take~$A_2= A'\cup L_1\cup C\cup L_2\cup B$ and observe that each of the hypergraphs~$A'$,~$L_1$,~$L_2$,~$C$, and~$B$ contain at most~$m'\ell\leq 2\binom{m}{3}\ell^2$ edges. 
    \end{proof}

\subsection{Proof of Lemma~\ref{lemma:absorbinglemma}}
\label{subsection:proofabsorbing}
We can finally give the short proof of Lemma~\ref{lemma:absorbinglemma}.

\begin{proof}[Proof of Lemma~\ref{lemma:absorbinglemma}]
    Given~$R\subseteq H$, an application of Lemma~\ref{lemma:tours} yields the existence of $A_1\subseteq H$ edge disjoint from~$R$ such that 
    	\begin{enumerate}[\upshape (i)]
	    \item $A_1$ has a~$C_\ell$-decomposition,
	    \item $A_1\cup R$ contain a tour decomposition, and
	    \item $A_1\cup R$ spans at most $30\binom{m}{3}\ell(6\ell+1)$ vertices.
	\end{enumerate}
    Then, we apply Lemma~\ref{lemma:tourstocycles} to obtain~$A_2\subseteq H$, which is an absorber of~$R\cup A_1$. 
    It is straightforward to check that~$A=A_1\cup A_2$ has the desired properties. 
\end{proof}

\section{Final remarks} \label{section:final}

A natural question is what happens for the values of $\ell$ not covered by our Theorem~\ref{theorem:ldecomposcodegree}.
Our results do not cover $C^3_{\ell}$-decompositions for small values of $\ell$, i.e. $\ell \leq 8$.
As in the graph case, for short cycles it is likely that the behaviour of the decomposition threshold is different.

For $\ell = 4$ the $3$-uniform tight cycle $C^3_{4}$ is isomorphic to a tetrahedron $K^3_4$, i.e. a complete $3$-graph on four vertices.
Since every pair of vertices in $K^3_4$ has degree $2$, the obvious necessary divisibility conditions in a host $3$-graph which admits a $C^3_4$-decomposition are \begin{enumerate*}[{\upshape (i)}]
        \item total number of edges divisible by $4$,
        \item every vertex degree divisible by $3$, and
        \item every codegree divisible by $2$.
\end{enumerate*}
Say that a $3$-graph satisfying all three conditions is \emph{$K^3_4$-divisible}.
We define $\delta_{\smash{K^3_{4}}}$ as the asymptotic minimum codegree threshold ensuring a $K^3_4$-decomposition over $K^3_4$-divisible graphs (in analogy to $\delta_{\smash{C_{\ell}}}$ taken over $C_{\ell}$-divisible graphs).
The following construction shows that $\delta_{\smash{K^3_{4}}} \ge 3/4$.

\begin{exmpl}
    Let $k \ge 1$ be arbitrary, $d = 6k+2$ and $n = 12k+9$.
    Let $G_1$ be an arbitrary $d$-regular graph on $n$ vertices.
    Let $G$ be the graph on $2n$ vertices obtained by taking two vertex-disjoint copies of $G_1$ and adding every edge between vertices belonging to different copies, say those edges are \emph{crossing}.
    Now, form a $3$-graph $H$ as follows.
    Take a set $Z$ on $2n$ vertices and edges forming a complete $3$-uniform graph on $Z$.
    Then add two new vertices $x_1, x_2$.
    For each $z \in Z$, add the edge $x_1 x_2 z$.
    Identify a copy of the graph $G$ in $Z$ and, for each edge $z_1 z_2$ of $G$ add the edges $z_1 z_2 x_1$ and $z_1 z_2 x_2$.
\end{exmpl}

$H$ has $2n+2 = 24k+20$ vertices and $\delta_2(H) = d+n+1 = 18k+12$ (attained by any pair $x_1 z$ with $z \in Z$).
It is tedious but straightforward to check $H$ is $K^3_4$-divisible.
To see $H$ is not $K^3_4$-decomposable, we prove that the link graph $H(x_1)$ is not~$C^2_3$-decomposable.
Note $H(x_1)$ is isomorphic to the graph $G'$ obtained from $G$ by adding an extra universal vertex $x$.
Suppose $G'$ has a triangle decomposition.
There are $n^2$ crossing edges in $G$, at most $n$ of those can be covered with triangles using $x$.
Thus at least $n(n-1)$ crossing edges are covered with triangles which use one edge in a copy of $G_1$ and two crossing edges.
Thus we need at least $n(n-1)/2$ edges in the two copies of $G_1$, but those copies have $dn < n(n-1)/2$ edges, contradiction.

What is the smallest $\ell_0$ such that $\delta_{\smash{C^3_{\ell}}} = 2/3$ holds for all $\ell \ge \ell_0$?
The previous example and Theorem~\ref{theorem:ldecomposcodegree} show that $5 \leq \ell_0 \leq 10^7$.
Observe that our Absorbing Lemma %
works for all $\ell \ge 7$.
The bottleneck is our use of Theorem~\ref{theorem:fractional} in
the Cover-Down Lemma. %
New ideas are needed to close the gap.

Another question is what happens for $k$-graphs with $k \ge 4$.
It is not clear for us if Theorem~\ref{theorem:counterexample} indicates the emergence of a pattern where the necessary codegree to ensure cycle decompositions and Euler tours on $n$-vertex $k$-graphs is substantially larger than $(1/2 + o(1))n$.

\begin{quest}
    For $k \ge 4$, let $H$ be a $k$-graph on $n$ vertices.
    Is $\delta_{k-1}(H) \geq  ((k-1)/k + o(1))n$ a necessary and sufficient condition for the existence of cycle decompositions or Euler tours?
\end{quest}

\subsection*{Acknowledgments}
We thank Felix Joos for helpful discussions and suggestions,
and the second author wants to thank Allan Lo and Vincent Pfenninger for useful conversations.

\printbibliography

\appendix

\section{Proof of Lemma~\ref{lemma:wellbehaved}} \label{appendix:wellbehaved}

\begin{proof}
	The proof proceeds in three steps.
	First, we find $H_p \subseteq H$ by including each edge with probability $p$, and in the remainder $H_0 = H \setminus H_p$ we find an almost perfect $C_{\ell}$-packing $\mathcal{C}_0$,
	let $L_0 = H_0 \setminus E(\mathcal{C}_0)$ be the leftover edges. 
	Secondly, we correct the leftover $L_0$ in the vertices incident with $\Omega(n^2)$ many edges of $L_0$ by constructing cycles with the help of the edges in $H_p$.
	This provides us with a new cycle packing $\mathcal{C}_1 \subseteq L_0 \cup H_p$ whose new leftover $L_1 = H_0 \setminus E(\mathcal{C}_0 \cup \mathcal{C}_1)$ satisfies $\Delta_1(L_1) = o(n^2)$.
	Finally, we correct the new leftover $L_1$ in a similar way, fixing the pairs incident to $\Omega(n)$ edges in $L_1$.
	We get a cycle packing $\mathcal{C}_2 \subseteq L_1 \cup H_p$,
	and $\mathcal{C}_0 \cup \mathcal{C}_1 \cup \mathcal{C}_2$ will be the desired cycle packing.
	\medskip
	
	\noindent \emph{Step 1: Random slice and aproximate decomposition.}
	Note that $\delta^{(3)}_2(H) \ge 3 \eps n$.
	Now let $p = \gamma / 4$, and let $H_p \subseteq H$ be obtained from $H$ by including each edge independently with probability $p$.
	Using concentration inequalities (e.g. Theorem~\ref{theorem:chernoff}) we see that with non-zero probability
	\begin{equation}
		\Delta_2(H_p) \leq 2 p n, \text{ and } \delta^{(3)}_2(H_p) \ge 2 \eps p n. \label{equation:Hpneighbours}
	\end{equation}
	hold simultaneously for $H_p$.
	From now on we suppose $H_p$ is fixed and satisfies~\eqref{equation:Hpneighbours}.
	
	Let $H_0 = H \setminus H_p$.
	In $H_0$, construct a $C_{\ell}$-packing by removing edge-disjoint cycles, one by one, until no longer possible.
	We get a $C_\ell$-packing $\mathcal{C}_0$ in $H_0$, let $F_0 = E(\cC_0)$.
	By Erd\H{o}s' Theorem~\cite[Theorem 1]{Erdos1964} there exists $c > 0$ such that $L_0 = H_0 \setminus F_0$ has at most $n^{3 - 3c}$ edges.
	\medskip
	
	\noindent \emph{Step 2: Eliminating bad vertices.}
	Let $B_0 = \{ v \in V : \deg_{L_0}(v) \ge n^{2 - 2c} \}$.
	Since $|L_0| \leq n^{3 - 3c}$, by double-counting we have $|B_0| \le 3 n^{1 - c}$.
	
	For each $b \in B_0$, let $G_b$ be the subgraph of $L_0(b)$ obtained after removing the vertices of $B_0$.
	Note that $L_0(b) - G_0$ has at most $|B_0|n \leq 3n^{2-c}$ edges.
	Now, let $\cP_b$ be a maximal edge-disjoint collection of paths of length $3$ in $G_b$.
	Since every graph on $n$ vertices with at least $n+1$ edges contains a path of length $3$,
	then $G_b - E(\cP_b)$ has at most $n$ edges.
	All together, we deduce that the number of edges in $L_0(b) - E(\cP_b)$ satisfies
	\begin{align}
		|L_0(b)| - |E(\cP_b)| \leq 3n^{2-c} + n \leq 4 n^{2-c}. \label{equation:wellbehaved-L0}
	\end{align}
	
	Since $G_b$ contains at most $n^2$ edges, we certainly have $|\cP_b| \leq n^2$.
	Let $\cP_b$ be a collection of tight paths on five vertices obtained by replacing each $v_0 v_1 v_2 v_3$ in $\cP_b$ with the tight path $v_0 v_1 b v_2 v_3$ in $L_0$.
	Note that any two distinct $P_1, P_2 \in \cP_b$ are edge-disjoint,
	and for two distinct $b, b' \in B_0$, and $P \in \cP_b$, $P' \in \cP_{b'}$, since $b' \notin V(G_b)$ we have $P, P'$ are edge-disjoint.
	Thus the union $\cP = \bigcup_{b \in B_0} \cP_b$ is an edge-disjoint collection of tight paths on $5$ vertices.
	
	Select $\gamma', \mu', \eps'$ such that $1/n \ll \gamma' \ll \mu' \ll \eps' \ll \gamma, \eps, 1/\ell$.
	We wish to apply Lemma~\ref{lemma:extending} to extend $\cP$ into cycles.
	We claim $\cP$ is $\gamma'$-sparse.
	Let $S \in \binom{V(H)}{2}$.
	Since $|\cP| \leq |B_0| n^2 \leq 3 n^{3 - c} \leq \gamma' n^{3}$,
	certainly $\cP$ contains at most $|\cP| \leq \gamma' n^{3}$ paths of type $0$ for $S$.
	Now, note that for each $b \in B_0$, $P \in \cP_b$ can have at most $2n$ paths of type $1$ for $S$,
	thus $\cP$ has at most $|B_0|2n \leq 6n^{2-c} \leq \gamma' n^{2}$ paths of type $1$ for $S$.
	Analogously, for each $b \in B_0$, $P \in \cP_b$ can have at most $1$ path of type $2$ for $S$,
	thus $\cP$ has at most $|B_0| \leq 3n^{1-c} \leq \gamma' n$ paths of type $2$ for~$S$.
	Thus $\cP$ is $\gamma'$-sparse.
	
	Recall that $L_0$ is edge-disjoint with $H_p$.
	Inequalities \eqref{equation:Hpneighbours} together with $p = \gamma/4$ and $\eps' \ll \gamma, 1/\ell$, show that we can use Corollary~\ref{corollary:pathextension} (with $U = V(H)$) and deduce that for each $P \in \cP$, there exists at least $\eps' n^{\ell - 5}$ copies of $C_{\ell}$ in $L_0 \cup H_p$ which extend $P_i$ using extra edges of $H_p$ only.
	
	We apply Lemma~\ref{lemma:extending} with $\eps', \mu', \gamma', \ell, 5, L_0, H_p, \cP$ playing the rôle of $\eps, \mu, \gamma, \ell, \ell', H_1, H_2, \cP$ respectively, to obtain a $C_{\ell}$-decomposable graph $F_1 \subseteq L_0 \cup H_p$ such that $E(\cP) \subseteq F_1$ and
	\begin{equation}
		\Delta_2(F_1 \setminus E(\cP)) \leq \mu'n. \label{item:wellbehaved-sparseC1}
	\end{equation}
	Since $F_0$, $F_1$ are edge-disjoint, $F_0 \cup F_1$ is $C_{\ell}$-decomposable.
	Let $L_1 = H_0 \setminus ( F_0 \cup F_1 )$. 
	Observe that, if $v \notin B_0$, then $\deg_{L_1}(v) \leq \deg_{L_0}(v) <n^{2-2c}$ by definition.
	Moreover, if $v \in B_0$, then each edge in $E(\cP_v)$ is in $F_1$, and hence~\eqref{equation:wellbehaved-L0} implies $\deg_{L_1}(v) \leq |L_0(v)| - |E(\cP_v)| \leq 4n^{2-c}$.
	Therefore,
	\begin{align}
		\Delta_1(L_1) \leq 4n^{2-c}. \label{equation:wellbehaved-L1vertexsparse}
	\end{align}
	\noindent \emph{Step 3: Eliminating bad pairs.}
	Let $f = c/2$ and $B_1 = \{ xy \in \binom{V}{2} : \deg_{L_1}(xy) \ge n^{1 - f} \}$.
	From $|L_1| \leq |L_0| \leq n^{3 - 3c} \leq n^{3 - 6f}$ we deduce $|B_1| \leq n^{2 - 4f}$.
	Now consider $B_1$ as the set of edges of a $2$-graph in $V$.
	Each edge of $B_1$ incident to a vertex $x$ implies that $x$ belongs to at least $n^{1 - f}$ edges in $L_1$, and each of those edges participates in at most two of the edges in $B_1$ incident to $x$.
	So we have $\deg_{L_1}(x) \ge \frac{1}{2} n^{1 - f} \deg_{B_1}(x)$.
	Together with inequality~\eqref{equation:wellbehaved-L1vertexsparse} we deduce $\Delta(B_1) \leq 8 n^{1 - f}$.
	
	A path $P$ on $L_1$ is \emph{$B_1$-based} if $P =zxyw$ and $xy \in B_1$.
	Let $\cP_2$ be a maximal packing of $B_1$-based paths.
	For all $xy \in B_1$, it holds that $\deg_{L_1}(xy) - \deg_{E(\cP_2)}(xy) \leq 1$.
	Otherwise it would exist distinct $z, w \in N_{L_1 \setminus E(\cP_2)}(xy)$, and then $zxyw$ would a $B_1$-based path not in~$\cP_2$ which contradicts its maximality.
	
	We claim $\cP_2$ is $\gamma'$-sparse.
	For each $xy \in B_1$, let $\cP_{xy} \subseteq \cP_2$ be the paths whose two interior vertices are precisely $xy$.
	Clearly $|\cP_{xy}| \leq n$ and $\cP_2 = \bigcup_{xy \in B_1} \cP_{xy}$.
	Let $e \in \binom{V}{2}$.
	Since $|\cP_2| \leq \sum_{xy \in B_1} |\cP_{xy}| \leq n |B_1| \leq n^{3 - 4f} \leq \gamma' n^3$,
	there are at most $\gamma' n^3$ paths of type $0$ for $e$ in $\cP_2$.
	Recall that if $P = zxyw$ is a path of type $1$ for $e$, then we have $|e \cap \{z, x, y, w\}| = 1$.
	If $xy \in B_1$ satisfies $e \cap \{x, y\} = \emptyset$, then at most two paths in $\cP_{xy}$ can be of type $1$ for $e$
	and therefore there are at most $2|B_1| \leq 2n^{2 - 4f}$ paths of type $1$ for $e$ in $\cP_2$.
	We estimate the contribution of the pairs $xy \in B_1$ such that $|e \cap \{x,y\}| = 1$.
	Each such $xy$ contributes with at most $n$ paths of type $1$ for $e$ in $\cP_{xy}$.
	By~\eqref{equation:wellbehaved-L1vertexsparse}, the number of such $xy$ is at most $2 \Delta(B_1) \leq 16 n^{1 - f}$,
	thus the total contribution of those pairs is at most $16 n^{2-f}$.
	All together, the total number of paths of type $1$ for $e$ in $\cP_2$ is at most $2n^{2 - 4f}+16 n^{2-f} \leq \gamma' n^2$.
	If $e = \{a,b\}$ then $\cP_{a,b}$ does not contain any path of type $2$ for $e$, by definition of the path types.
	Thus the only possible contributions come from the pairs in $\cP_{a,x}$ and $\cP_{b,y}$ for some $x, y \in V(H)$; and each one of those sets contains at most $1$ path of type $2$ for $e$.
	Thus the total number of pairs of type $2$ for $e$ in $\cP_2$ is at most $2 \Delta(B_1) \leq 16 n^{1 - f} \leq \gamma' n$.
	Thus $\cP_2$ is $\gamma'$-sparse.
	
	Let $H'_p = H_p \setminus ( F_0 \cup F_1 )$.
	\eqref{equation:Hpneighbours} and~\eqref{item:wellbehaved-sparseC1}, together with $\mu' \ll \eps' \ll \gamma, 1/\ell$,
	allow us to use Corollary~\ref{corollary:pathextension} with $U = V(H)$,
	thus for each $P \in \cP_2$, there exists at least $\eps' n^{\ell - 4}$ copies of $C_{\ell}$ in $L_1 \cup H'_p$ which extend $P$ using extra edges of $H'_p$ only.
	Apply Lemma~\ref{lemma:extending} with the parameters $\eps', \mu', \gamma', \ell, 4, L_1, H'_p, \cP_2$ playing the rôles of $\eps, \mu, \gamma, \ell, \ell', H_1, H_2, \cP$ respectively, to obtain a $C_{\ell}$-decomposable $F_2 \subseteq L_1 \cup H'_p$ such that $E(\cP_2) \subseteq F_2$
	and	$\Delta_2(F_2 \setminus E(\cP_2)) \leq \mu'n$.
	
	We claim that $\Delta_2(L_1 \setminus F_2) \leq n^{1 - f}$.
	Indeed, if $xy \in B_1$, $\deg_{L_1 \setminus F_2}(xy) \leq \deg_{L_1}(xy) \leq n^{1-f}$ follows by definition,
	otherwise, $E(\cP_2) \subseteq F_2$ implies $\deg_{L_1 \setminus F_2}(xy) \leq \deg_{L_1}(xy) - \deg_{F_2}(xy) \leq 1$.
	Since $F_2$ and $F_0 \cup F_1$ are edge-disjoint, $F = F_0 \cup F_1 \cup F_2$ is a $C_{\ell}$-decomposable subgraph of $H$.
	We claim $L = H \setminus F$ satisfies $\Delta_2(L) \leq \gamma n$.
	Indeed, an edge not covered by $F$ is either in $H_p$ or in $L_1 \setminus F_2$.
	Thus we have $\Delta_2(L) \leq \Delta_2(H_p) + \Delta_2(L_1 \setminus F_2) \leq 2pn + n^{1-f} \leq \gamma n$,
	as required.
\end{proof}

\end{document}